\documentclass[11pt]{article}

\usepackage[T1]{fontenc}
\usepackage[utf8]{inputenc}
\usepackage{lmodern}
\usepackage[a4paper,margin=1in]{geometry}
\usepackage{amsmath,amssymb,amsthm,mathtools}
\usepackage{mathrsfs}
\usepackage{booktabs}
\usepackage{microtype}
\usepackage[colorlinks=true,linkcolor=blue,citecolor=blue,urlcolor=blue]{hyperref}

\newtheorem{theorem}{Theorem}[section]
\newtheorem{lemma}[theorem]{Lemma}
\newtheorem{proposition}[theorem]{Proposition}
\newtheorem{corollary}[theorem]{Corollary}
\theoremstyle{definition}
\newtheorem{definition}[theorem]{Definition}
\theoremstyle{remark}
\newtheorem{remark}[theorem]{Remark}

\newcounter{maintheorem}
\renewcommand{\themaintheorem}{\arabic{maintheorem}}
\newenvironment{maintheorem}
{\refstepcounter{maintheorem}\par\medskip
 \noindent{\bf Theorem \themaintheorem.}\it}
{\par\medskip}

\newcommand{\Z}{\mathbf Z}

\newcommand{\R}{\mathbf R}
\newcommand{\F}{\mathbf F}

\newcommand{\Aut}{\operatorname{Aut}}
\newcommand{\rk}{\operatorname{rk}}
\newcommand{\rank}{\operatorname{rank}}

\newcommand{\eps}{\varepsilon}
\newcommand{\ip}[2]{\left(#1,#2\right)}

\newcommand{\one}{\mathbf 1}

\title{The unique extremal threemodular lattice of rank 26, the generalized hexagon $(2,8)$, and the tight Cayley-plane $5$-design}
\author{Gerald H\"ohn\thanks{Department of Mathematics, Kansas State University.}}
\date{}

\begin{document}
\maketitle

\begin{abstract}
We reconstruct the generalized hexagon of order $(2,8)$ from an abstract even rank-$26$ lattice of determinant $3$ and minimum $4$, and conversely reconstruct the lattice from the hexagon.  Harmonic theta identities determine the $819$ shortest vectors in a non-zero discriminant class and their association scheme, and show that every non-empty positive-norm shell of the lattice and its dual is a spherical $5$-design.  In the converse direction, the rank-$26$ idempotent gives the lattice, with saturation proved by a short dual-coset argument.  A positive-definite Niemeier construction proves existence and uniqueness of the lattice and hence of the hexagon.

The third moment of a projective $3$-design in the Cayley plane reconstructs the traceless Albert product.  This upgrades angle-preserving bijections to elements of $F_4(\mathbb R)$ and proves geometric uniqueness of the tight $819$-point projective $5$-design.  Finally, line deletion gives a rootless index-four sublattice of $N(A_1^{24})$ with a Golay trio and an oriented gluing.  The gluing data form eight root-sign orbits, on which $L_3(2)\cong L_2(7)$ acts as on $\mathbf P^1(\F_7)$; two explicit Golay permutations prove transitivity.  This proves uniqueness and gives the automorphism-group order.  The group is identified afterward as ${}^3D_4(2):3$, with a central factor $C_2$ for the full lattice group.  The reconstruction and deletion arguments parallel the length-$26$ binary code construction.
\end{abstract}

\section{Introduction}

A generalized hexagon is a point-line incidence geometry whose incidence graph has diameter $6$ and girth $12$.  It has order $(s,t)$ if every line contains $s+1$ points and every point is incident with $t+1$ lines.  The numbers of points and lines are then
\[
        (s+1)(s^2t^2+st+1),
        \qquad
        (t+1)(s^2t^2+st+1),
\]
respectively.  For $(s,t)=(2,8)$ these numbers are
\[
        819,
        \qquad
        2457.
\]
The classical example is the Tits triality hexagon associated with the Steinberg group ${}^3D_4(2)$; its full automorphism group is ${}^3D_4(2):3$ \cite{Steinberg1959,Tits1959,CarterLie,TitsBuildings}.  The finite-geometric uniqueness of this hexagon was proved by Cohen and Tits \cite{CohenTits}.  The aim of this paper is to give a different proof of this uniqueness, passing through the rank-$26$ determinant-$3$ lattice isolated by Borcherds, and to develop an intrinsic reconstruction parallel to the author's length-$26$ binary code proof \cite{HohnCode26}.  This is close in spirit to the rigidity arguments for spherical designs and association schemes developed by Bannai, Munemasa, and Venkov, and by Bannai--Bannai--Bannai \cite{BannaiMunemasaVenkov,BannaiBannaiBannai}.

On the lattice side, let
\begin{equation}\label{eq:initial-L}
        L,
        \qquad \rk L=26,
        \qquad \det L=3,
        \qquad \min L=4,
\end{equation}
be an even positive-definite integral lattice.  Its discriminant group is cyclic of order $3$; we write
\[
        L^\#/L=\{0,C_+,C_-\},
        \qquad C_-=-C_+.
\]
Borcherds proved that there is a unique such lattice, equivalently a unique rank-$27$ odd unimodular rootless lattice together with a characteristic vector of norm $3$, and he computed its automorphism-group order \cite[Sec.~5.7]{BorcherdsThesis}.  We give a positive-definite proof of both conclusions, using the Niemeier lattice with root system $A_1^{24}$ and its Golay glue code in place of the Lorentzian classification.  Following his terminology, the title uses \emph{threemodular} (his thesis writes \emph{trimodular}) for this genus of even rank-$26$ lattices of determinant $3$; this is not the modern similarity condition sometimes also called $3$-modularity, which in rank $26$ would force determinant $3^{13}$.

The lattice is already implicit in the ATLAS construction.  As Borcherds pointed out to the author (personal communication), the ATLAS gives a $27$-dimensional representation of $2\times{}^3D_4(2):3$ containing a norm-$3$ vector $c$ fixed by ${}^3D_4(2):3$ and an orbit of $819$ norm-$1$ vectors having inner product $1$ with $c$ \cite[p.~89]{Atlas}; the integral lattice generated by $c$ and twice these $819$ vectors is the exceptional odd unimodular rank-$27$ lattice, and the determinant-$3$ lattice in \eqref{eq:initial-L} is recovered from the orthogonal complement of $c$ \cite[Sec.~5.7]{BorcherdsThesis}.  Thus the representation and orbit were already present in the ATLAS, while Borcherds' work established the lattice uniqueness that we reprove below.

Bacher--Venkov later placed the associated rank-$27$ rootless unimodular lattice in their neighbour classification \cite{BacherVenkov}; King's mass formula verifies the rank-$27$ and rank-$28$ enumerations \cite{King}; and Chenevier's later classification of integral unimodular lattices in ranks $26$ and $27$ provides another check through the rank-$27$ correspondence \cite{ChenevierHunting}.  Elkies and Gross studied the determinant-$3$ lattice in the exceptional cone, or equivalently in the integral Albert-algebra model, and recorded the $819$ pairs of norm-$8/3$ vectors in the dual lattice \cite{ElkiesGrossCone}.  Their later paper on cubic rings belongs to the same Albert-algebra programme \cite{ElkiesGrossCubic}.  The lattice--hexagon equivalence below uses neither cubic rings nor Jordan multiplication; the Albert product enters only in Section~\ref{sec:cayley-design}, where it is reconstructed from the third moment of a tight Cayley-plane design in order to upgrade an ambient orthogonal equivalence to an $F_4(\mathbb R)$-equivalence.

The finite-geometric bridge proved here is the following: in either non-zero discriminant class of an abstract lattice satisfying \eqref{eq:initial-L}, the shortest vectors form exactly the point set of the generalized hexagon of order $(2,8)$, and every such hexagon reconstructs the lattice.  The first point needing care is that the non-zero discriminant classes could a priori have contained vectors of norm $2/3$.  Section~\ref{sec:harmonics} rules this out from \eqref{eq:initial-L}.  Thus
\begin{equation}\label{eq:coset-extremal}
        \min C_+=\min C_- = \frac83
\end{equation}
is a consequence, not an additional hypothesis.

The lattice--hexagon construction is deliberately intrinsic.  We do not start from the octonion projective plane, the exceptional Jordan algebra, the Leech lattice, or the ATLAS representation.  We choose one non-zero discriminant class and set
\begin{equation}\label{eq:X-def-intro}
        X=(C_+)_{8/3}=\{x\in C_+:x^2=8/3\}.
\end{equation}
The other non-zero class has shortest shell $-X$.  Thus the two non-zero classes contain $1638$ oriented shortest vectors, or $819$ projective points.

\begin{maintheorem}\label{thm:main-forward}
Let $L$ satisfy \eqref{eq:initial-L}.  Then the non-zero discriminant classes satisfy \eqref{eq:coset-extremal}, and $X=(C_+)_{8/3}$ has $819$ elements.  Every non-empty positive-norm shell of $L$ and $L^\#$ is a spherical $5$-design.  The three inner-product relations
\[
        \ip{x}{y}=-\frac43,
        \qquad
        \ip{x}{y}=\frac23,
        \qquad
        \ip{x}{y}=-\frac13
\]
on distinct elements of $X$ form a three-class association scheme with first eigenmatrix
\begin{equation}\label{eq:intro-P}
 P=
 \begin{pmatrix}
 1&18&288&512\\
 1&5&2&-8\\
 1&-3&-6&8\\
 1&-9&72&-64
 \end{pmatrix}.
\end{equation}
The relation $\ip{x}{y}=-4/3$ has lines
\[
        \{x,y,z\}\subset X,
        \qquad x+y+z=0,
\]
and these points and lines form a generalized hexagon of order $(2,8)$.
\end{maintheorem}

Recovering a geometry from the shortest shell does not yet show that the geometry determines the lattice.  In the converse direction the rank-$26$ Bose--Mesner eigenspace supplies the Euclidean representation.  The essential arithmetic point is to prove that its integral span is saturated in the required dual lattice.

\begin{maintheorem}\label{thm:main-converse}
Conversely, every generalized hexagon of order $(2,8)$ produces, from its rank-$26$ Bose--Mesner idempotent, an even integral lattice $L_H$ with
\[
        \rk L_H=26,
        \qquad \det L_H=3,
        \qquad \min L_H=4.
\]
The shortest vectors in one non-zero discriminant class of $L_H^\#/L_H$ recover the original hexagon.
\end{maintheorem}

The two constructions are inverse, so lattice uniqueness can replace a finite-geometric uniqueness argument.  Their functoriality also identifies the automorphism groups abstractly.  Section~\ref{sec:niemeier-descent} proves lattice uniqueness and calculates the group order from positive-definite deletion data.  Combining this with the group recognition at the end of the paper gives the following result.

\begin{maintheorem}\label{thm:main-unique}
There exists, up to isometry, exactly one lattice satisfying \eqref{eq:initial-L}, and the generalized hexagon of order $(2,8)$ is unique.  Its automorphism group is
\[
        {}^3D_4(2):3,
\]
and the full automorphism group of the rank-$26$ lattice is
\[
        \Aut(L)\cong C_2\times({}^3D_4(2):3).
\]
\end{maintheorem}

The existence assertion is constructive: an admissible quotient of $N(A_1^{24})$ and a compatible gluing produce the lattice and hence its hexagon entirely in the positive-definite setting.  Neither the ATLAS representation nor an a priori Albert-algebra model is needed for this construction.

The same $819$-point geometry has a classical realization in the Cayley plane as a tight projective design.  We also prove uniqueness of this realization under the automorphism group of the Albert algebra, rather than only under the orthogonal group of its trace-zero space.

\begin{maintheorem}\label{thm:main-cayley}
Let $J=\operatorname{Herm}_3(\mathbb O)$ be the compact real Albert algebra and let $\mathbb{O}P^2$ be its manifold of primitive trace-one idempotents.  Every tight projective $5$-design in $\mathbb{O}P^2$ is conjugate under
\[
        \Aut(J)\cong F_4(\mathbb R)
\]
to the classical $819$-point design.  Its setwise stabilizer in $F_4(\mathbb R)$ is
\[
        {}^3D_4(2):3.
\]
\end{maintheorem}

The organization is parallel to the author's length-$26$ code proof \cite{HohnCode26}.  There, harmonic weight enumerators and the minimal half-shadows recover the projective plane of order $3$, while deletion of an intrinsic flag gives the odd Golay code together with a deep-hole coset.  Here harmonic theta series and the shortest discriminant shells lead to the generalized hexagon, while deletion of a line leads to $N(A_1^{24})$ with a Golay trio, a quotient map, and a discriminant gluing.  The essential correspondences are:
\begin{center}
\small
\renewcommand{\arraystretch}{1.16}
\begin{tabular}{@{}p{0.445\linewidth}@{\hspace{0.045\linewidth}}p{0.49\linewidth}@{}}
\toprule
Length-$26$ binary code & Rank-$26$ lattice\\
\midrule
Harmonic weight enumerators & Harmonic theta series\\
Minimal half-shadows & Shortest non-zero discriminant shells\\
Projective plane $\mathrm{PG}(2,3)$ & Generalized hexagon of order $(2,8)$\\
Deletion of an intrinsic flag & Deletion of an intrinsic line\\
Odd Golay code and a deep-hole coset & $N(A_1^{24})$, a Golay trio, and a quotient/gluing datum\\
One orbit of deep-hole extension data & One orbit of oriented Niemeier deletion data\\
\bottomrule
\end{tabular}
\end{center}
In both constructions the smaller object alone is insufficient: the retained extension datum is what permits reconstruction.  Its single-orbit property proves uniqueness and determines the automorphism-group order.

This comparison is part of the author's general philosophy of studying analogies between codes, lattices, and vertex operator algebras.

Section~\ref{sec:harmonics} gives the short-coset exclusion, the harmonic theta identities, the all-shell $5$-design property, and the $E_6$-gluing calculation of the theta characters, in the Venkov--Bacher--Bachoc framework \cite{VenkovSurvey,BachocVenkov,BacherVenkov}.  Section~\ref{sec:819-hexagon} proves that the $819$ vectors form the association scheme and the generalized hexagon.  Section~\ref{sec:converse} constructs the lattice back from the hexagon, proves saturation by a short dual-coset argument, and records the abstract correspondence of automorphism groups and the consequences of the lattice uniqueness proved in Section~\ref{sec:niemeier-descent}.  Section~\ref{sec:cayley-design} proves geometric uniqueness of the tight Cayley-plane design.  Finally, Section~\ref{sec:niemeier-descent} describes line deletion and reconstruction, proves transitivity on the admissible rank-$24$ gluing data, and thereby proves lattice uniqueness and obtains the automorphism-group order.  Only at the end do we use the Steinberg--Tits construction to identify the group by name and discuss its local subgroups.

The external inputs separate according to these tasks.  Beyond standard discriminant-form and theta-series theory, the lattice argument uses low-weight modular-form dimensions, the required $E_6,E_7$ shell data, and the basic Golay--Mathieu facts about code uniqueness, the Mathieu order, and trios \cite{ConwaySloane,VenkovSurvey,BachocVenkov}.  The specific trio action needed below is derived from the displayed Golay generators.  The Cayley-plane argument additionally uses the standard Albert-algebra/$F_4$ structure, its projective moment formulae, Hoggar's bound and angle set, and the existence of the classical tight design \cite{SpringerVeldkamp,Hoggar1982,Hoggar1984,Hoggar1989,Nasmith2022}.  Steinberg--Tits and the ATLAS are proof inputs only for the final group identification and the finer local structure.  Neither Borcherds' Lorentzian uniqueness theorem, Cohen--Tits uniqueness, nor the Bacher--Venkov rank-$27/28$ classification is used to prove lattice uniqueness or calculate the group order.  Leech-lattice uniqueness is used only in the optional neighbour remark in Section~\ref{sec:niemeier-descent}.

The Cayley-plane conclusion explains why the design structure is more than another realization of the hexagon.  Hoggar's construction and Nasmith's octonionic formulation give the classical tight $819$-point projective $5$-design \cite{Hoggar1982,Hoggar1984,Hoggar1989,Nasmith2022}.  Abstract uniqueness determines its centered Gram matrix and hence an equivalence under $O(26)$, but not yet under $F_4(\mathbb R)$.  The third projective moment recovers the traceless Albert product and forces the orthogonal equivalence to preserve the Cayley plane.  This last rigidity step already holds for projective $3$-designs.

\section{Harmonics, designs, and theta characters}
\label{sec:harmonics}

Let $L$ satisfy \eqref{eq:initial-L}.  Since $L$ is even and $\det L=3$, the discriminant form on $A_L=L^\#/L$ is cyclic of order $3$.  The signature congruence for finite quadratic forms \cite[Ch.~4]{ConwaySloane} forces the non-zero values to be $2/3$ modulo $2\Z$ (the normalized Gauss sum is $i$, since $26\equiv2\pmod8$).  Thus
\[
        q(C_+)=\frac23\pmod{2\Z},
        \qquad q(C_-)=\frac23\pmod{2\Z}.
\]
Vectors in the two non-zero classes therefore have norms congruent to $2/3$ modulo $2\Z$.  We will exclude the first congruent norm by a positive-definite root lemma, so that the first possible non-zero class norm is $8/3$.

We first record the modular identities used both for the short-coset exclusion and for the $E_6$-gluing calculation.  They are the even-unimodular rank-$32$ counterpart of the Bacher--Venkov harmonic identities.

\begin{lemma}[rank-$32$ scalar and harmonic theta identities]\label{lem:rank32-BV}
Let $N$ be an even unimodular lattice of rank $32$, let $N_2$ and $N_4$ denote its norm-$2$ and norm-$4$ shells, and let $P$ be a homogeneous harmonic polynomial on $N\otimes\R$.
\begin{enumerate}
\item The scalar theta series gives
\begin{equation}\label{eq:rank32-scalar}
        |N_4|=146880+216|N_2|.
\end{equation}
\item If $P$ has degree $2$, then
\begin{equation}\label{eq:rank32-degree2}
        \sum_{z\in N_4}P(z)=-528\sum_{z\in N_2}P(z).
\end{equation}
\item If $P$ has degree $4$ or $6$ and $\sum_{z\in N_2}P(z)=0$, then
\begin{equation}\label{eq:rank32-degree46-zero}
        \sum_{z\in N_4}P(z)=0.
\end{equation}
\end{enumerate}
\end{lemma}

\begin{proof}
With the convention $\theta_N(q)=\sum_{z\in N}q^{z^2/2}$, the scalar theta series of a rank-$32$ even unimodular lattice is
\[
        E_4^4+(|N_2|-960)\Delta E_4.
\]
The coefficient of $q^2$ is therefore $146880+216|N_2|$.  If $P$ is harmonic of degree $d$, then the weighted theta series $\sum_{z\in N}P(z)q^{z^2/2}$ is a cusp form of weight $16+d$ for $\mathrm{SL}_2(\Z)$.  For $d=2,4,6$ the relevant cusp spaces are respectively spanned by
\[
        \Delta E_6,
        \Delta E_8,
        \Delta E_{10}.
\]
Since
\[
        \Delta E_6=q-528q^2+\cdots,
\]
we get \eqref{eq:rank32-degree2}.  For $d=2,4,6$, a zero $q$-coefficient forces the entire weighted theta series to vanish, since each displayed generator has leading coefficient $1$.  In particular this gives \eqref{eq:rank32-degree46-zero}.
\end{proof}

The following root lemma is also a consequence of Borcherds' Lorentzian classification \cite[Lemma~2.5]{BorcherdsClass}.  The proof here uses only positive-definite gluing and the preceding modular identities.

\begin{lemma}[rank-$25$ determinant-$2$ root lemma]\label{lem:rank25-det2-root}
Every even positive-definite lattice of rank $25$ and determinant $2$ contains a vector of norm $2$.
\end{lemma}

\begin{proof}
Suppose that $K$ is such a lattice without roots.  Its non-zero discriminant class has norm $1/2$ modulo $2\Z$, by the signature congruence.  It contains no vector of norm $1/2$, since twice such a vector would be a root of $K$.  Thus its minimum is at least $5/2$.

Glue $K$ to $E_7$ along their discriminant groups to obtain an even unimodular lattice $N$ of rank $32$.  The non-zero $E_7$ class has minimum $3/2$, so the roots of $N$ are exactly the $126$ roots of $E_7$.  Write
\[
 Y=(K^\#\setminus K)_{5/2},\qquad a=|Y|,\qquad b=|K_4|.
\]
The standard $E_7$ shell counts are $|(E_7)_4|=756$ and $|(E_7^\#\setminus E_7)_{3/2}|=56$; they follow from the root-lattice theta series \cite[Ch.~4]{ConwaySloane}.  Hence the norm-$4$ shell of $N$ is the disjoint union of $K_4$, $(E_7)_4$, and the products of these $56$ dual vectors with $Y$.

Let $s$ and $t$ be the squared norms of the projections to $K\otimes\R$ and $E_7\otimes\R$, respectively.  The scalar identity and the degree-two identity for the harmonic polynomial $25t-7s$ give
\[
 b+56a=173340,\qquad b-40a=121500.
\]
For the second equation, the sums on $N_2$ and $N_4$ are respectively $6300$ and $75600+1120a-28b$.  It follows that
\[
 a=540,\qquad b=143100.
\]
In particular, $Y$ is non-empty.

We next show that $Y$ is a spherical $2$-design.  Let $h$ be any homogeneous harmonic quadratic polynomial on $K\otimes\R$, extended independently of the $E_7$ variables.  Both $h$ and
\[
 (29t-7s)h
\]
are harmonic on $N\otimes\R$ and vanish on its roots.  Indeed, $\Delta(th)=14h$ and $\Delta(sh)=58h$.  Put $A_h=\sum_{y\in Y}h(y)$ and $B_h=\sum_{k\in K_4}h(k)$.  Lemma~\ref{lem:rank32-BV}, in degrees two and four, gives
\[
 B_h+56A_h=0,\qquad -28B_h+26\cdot56A_h=0.
\]
Thus $A_h=0$ for every $h$.  The shell $Y$ is antipodal, so this is precisely the required $2$-design property.

Choose $y_0\in Y$.  If $y\ne\pm y_0$, then $y_0\pm y$ are non-zero vectors of $K$, so their norms are at least $4$.  Moreover, $y_0-y\in K$ and $y_0\in K^\#$ give
\[
 (y_0,y)\equiv y_0^2=\frac52\pmod{\Z},
 \qquad (y_0,y)\in\frac12+\Z.
\]
Together with the norm bounds, this forces $(y_0,y)=\pm1/2$.  The $2$-design identity would give
\[
 \sum_{y\in Y}(y_0,y)^2
       =\frac{|Y|}{25}\left(\frac52\right)^2
       =\frac{|Y|}{4}.
\]
But the two terms $y=\pm y_0$ and the remaining terms give instead
\[
 2\left(\frac52\right)^2+\frac{|Y|-2}{4}
       =\frac{|Y|}{4}+12,
\]
a contradiction.
\end{proof}

\begin{proposition}[exclusion of norm $2/3$ in the non-zero classes]\label{prop:no-two-thirds}
Let $L$ satisfy \eqref{eq:initial-L}.  Then
\[
        (C_+)_{2/3}=(C_-)_{2/3}=\varnothing .
\]
In particular, $\min C_\pm\ge8/3$.
\end{proposition}

\begin{proof}
It is enough to treat one non-zero class.  Suppose that $x\in C_+$ and $x^2=2/3$.  Then
\[
        v=3x\in L,
        \qquad v^2=6.
\]
The vector $v$ is primitive: if $v=mu$ in $L$ with $m>1$, then $u^2=6/m^2\le 3/2$, impossible in the positive even lattice $L$.  Moreover $(v,L)\subset 3\Z$, because $x\in L^\#$.  Since $v$ is primitive, the class $v/\operatorname{div}_L(v)$ has order $\operatorname{div}_L(v)$ in $A_L$.  As $A_L$ has order $3$, and as $(v,L)\subset 3\Z$, the divisibility of $v$ in $L$ is exactly $3$.

Let
\[
        K=v^\perp\cap L .
\]
Then $K$ is even, positive definite, has rank $25$, and has minimum at least $4$, since $K\subset L$.  The standard determinant formula for the orthogonal complement of a primitive vector gives
\[
        \det K
        =\det L\,\frac{v^2}{\operatorname{div}_L(v)^2}
        =3\,\frac{6}{3^2}=2 .
\]
This contradicts Lemma~\ref{lem:rank25-det2-root}.  Thus no vector of norm $2/3$ exists in a non-zero discriminant class.  The next possible norm in such a class is $8/3$; its occurrence, and hence \eqref{eq:coset-extremal}, will follow from the theta calculation in Proposition~\ref{prop:BBV-package}.
\end{proof}

We also need the first terms of the ordinary and coset theta series of $E_6$.  Let $R_6=(E_6)_2$, $T_6=(E_6)_4$, and let $Y$ be one of the two non-zero minuscule classes of $E_6^\#/E_6$.  Then
\begin{equation}\label{eq:E6-counts}
        |R_6|=72,
        \qquad |T_6|=270,
        \qquad |Y_{4/3}|=27.
\end{equation}
These numbers are read off from the $E_6$ theta series and from the minuscule coset theta series.  Equivalently, they follow by a direct enumeration in the root lattice.
We use the standard normalization in which the roots have norm $2$; see, for example, the root-lattice and glue conventions in \cite[Ch.~4]{ConwaySloane}.

\begin{proposition}[rank-$32$ even moment package]\label{prop:BBV-package}
Let $L$ be as in \eqref{eq:initial-L}.  For each non-zero discriminant class $C$ one has
\[
        \min C=\frac83,
        \qquad |C_{8/3}|=819.
\]
For $X=C_{8/3}$, the following identities hold in $V=L\otimes_\Z\R$:
\begin{align}
        \sum_{x\in X}\ip{a}{x}\ip{b}{x}
        &=84\ip{a}{b},
                &&a,b\in V,                                    \label{eq:BBV2}\\
        \sum_{x\in X}\ip{a}{x}\ip{b}{x}\ip{c}{x}\ip{d}{x}
        &=8\bigl(\ip{a}{b}\ip{c}{d}
                 +\ip{a}{c}\ip{b}{d}
                 +\ip{a}{d}\ip{b}{c}\bigr),
                &&a,b,c,d\in V.                                \label{eq:BBV4tensor}
\end{align}
In particular,
\begin{equation}\label{eq:BBV4}
        \sum_{x\in X}\ip{a}{x}^2\ip{b}{x}^2
        =8\bigl(\ip{a}{a}\ip{b}{b}+2\ip{a}{b}^2\bigr).
\end{equation}
\end{proposition}

\begin{proof}
Choose an anti-isometry $A_L\to A_{E_6}$ and form the even unimodular rank-$32$ lattice
\begin{equation}\label{eq:N-proof-glue}
        N=L\oplus_\phi E_6
        =\{(u,v)\in L^\#\oplus E_6^\#:\phi(\bar u)=\bar v\}.
\end{equation}
Let $C$ be one non-zero class of $L^\#/L$, and put $n=|C_{8/3}|$ and $a=|L_4|$.  Proposition~\ref{prop:no-two-thirds} shows that the root system of $N$ is exactly $E_6$.  The norm-$4$ shell consists of $L_4$, $(E_6)_4$, and the two mixed shells $C_{8/3}\times Y_{4/3}$ and $(-C)_{8/3}\times(-Y)_{4/3}$.  Hence
\[
 |N_2|=72,\qquad |N_4|=a+270+54n.
\]
The scalar identity \eqref{eq:rank32-scalar} gives
\begin{equation}\label{eq:first-count-equation}
 a+54n=162162.
\end{equation}

For $u\in L\otimes\R$ and $v\in E_6\otimes\R$, the quadratic polynomial $Q(u,v)=3u^2-13v^2$ is harmonic in dimension $26+6$.  Its shell sums are
\[
 \sum_{N_2}Q=-1872,\qquad
 \sum_{N_4}Q=12a-504n-14040.
\]
Using \eqref{eq:rank32-degree2}, we obtain $12a-504n=1002456$.  Together with \eqref{eq:first-count-equation}, this gives
\[
 |L_4|=117936,\qquad |C_{8/3}|=819.
\]
In particular, $\min C=8/3$.

We next extract the even harmonic moments of $X$.  Let $H\in\mathcal H_d(V)$ be harmonic of degree $d=2$ or $4$, and put
\[
        A_H=\sum_{u\in L_4}H(u),
        \qquad
        B_H=\sum_{x\in X}H(x).
\]
First use the harmonic polynomial $P_0(u,v)=H(u)$ on the rank-$32$ space.  Its root-shell sum is zero, since all roots of $N$ lie in the $E_6$ factor.  Hence \eqref{eq:rank32-degree2} or \eqref{eq:rank32-degree46-zero}, according as $d=2$ or $4$, gives
\begin{equation}\label{eq:AHBH-one}
        A_H+54B_H=0.
\end{equation}
Now put
\begin{equation}\label{eq:P1-proof}
        P_1(u,v)=H(u)\left(v^2-\frac{3}{d+13}u^2\right).
\end{equation}
The identity
\[
        \Delta_u(u^2H)=(4d+52)H,
        \qquad
        \Delta_v(v^2)=12
\]
shows that $P_1$ is harmonic on the full rank-$32$ space.  Again its root-shell sum is zero.  On $L_4$ it has value
$-4\frac{3}{d+13}H$, and on the mixed shell
$X\times Y\sqcup(-X)\times(-Y)$ it has value
$\bigl(\frac43-\frac{8}{3}\frac{3}{d+13}\bigr)H$ in each of the $54$ minuscule directions.  Thus
\begin{equation}\label{eq:AHBH-two}
        -\frac{12}{d+13}A_H
        +54\left(\frac43-\frac{8}{d+13}\right)B_H=0.
\end{equation}
Equations \eqref{eq:AHBH-one} and \eqref{eq:AHBH-two} are independent; substituting $A_H=-54B_H$ into \eqref{eq:AHBH-two} gives
$72(1+3/(d+13))B_H=0$.  Hence
\begin{equation}\label{eq:even-harmonic-vanishing}
        \sum_{x\in X}H(x)=0
        \qquad(H\in\mathcal H_2(V)\oplus\mathcal H_4(V)).
\end{equation}
Since all vectors of $X$ have norm $8/3$ and $|X|=819$, the usual spherical moment formulae in dimension $26$ give
\[
        \sum_{x\in X}\ip{a}{x}\ip{b}{x}
        =819\frac{8/3}{26}\ip{a}{b}=84\ip{a}{b}
\]
and, by polarization of the degree-$4$ formula,
\[
        \sum_{x\in X}\ip{a}{x}\ip{b}{x}\ip{c}{x}\ip{d}{x}
        =819\frac{(8/3)^2}{26\cdot28}
        \sum_{\text{pairings}}\ip{\cdot}{\cdot}\ip{\cdot}{\cdot},
\]
which is exactly \eqref{eq:BBV4tensor}, because
$819(8/3)^2/(26\cdot28)=8$.
\end{proof}

The vanishing used in the preceding proof holds for the entire weighted theta series.  Keeping all its coefficients gives a stronger conclusion, by the same low-weight modular-form mechanism used in \cite{VenkovSurvey,BachocVenkov}.

\begin{theorem}[all shells are spherical $5$-designs]\label{thm:all-shell-designs}
Every non-empty shell of positive norm in $L$ or $L^\#$ is a spherical $5$-design.  More precisely, for $d=2,4$ and $H\in\mathcal H_d(V)$, the sum of $H$ over every shell of $L$, $C_+$, or $C_-$ vanishes.
\end{theorem}

\begin{proof}
Use the gluing $N$ of \eqref{eq:N-proof-glue} and the two harmonic polynomials $P_0(u,v)=H(u)$ and $P_1(u,v)$ in \eqref{eq:P1-proof}.  Their weighted theta series have weights $16+d$ and $18+d$, respectively.  These are among $18,20,22$, whose cusp spaces are one-dimensional with generators having non-zero $q$-coefficients.  Both root-shell sums vanish because all roots lie in the $E_6$-factor.  Thus both weighted theta series vanish identically.

For $k\ge1$ put
\[
 A_k=\sum_{u\in L_{2k}}H(u),\qquad
 B_k=\sum_{u\in(C_+)_{2k-4/3}}H(u),
\]
with an empty sum interpreted as zero.  The corresponding sum on $C_-$ equals $B_k$, since $d$ is even.  In the coefficient of $q^k$, terms from $L\times E_6$ involve $A_i$ with $i\le k$, and terms from the two non-zero gluing classes involve $B_i$ with $i\le k$, since the $E_6$ cosets have minimum $4/3$.  After the lower sums have vanished inductively, only $L_{2k}\times\{0\}$ and the two mixed shells with $E_6$-norm $4/3$ remain.  Their contributions give
\begin{equation}\label{eq:all-shell-system}
 \begin{pmatrix}
 1&54\\[1mm]
 -\dfrac{6k}{d+13}&
 54\left(\dfrac43-\dfrac{6k-4}{d+13}\right)
 \end{pmatrix}
 \binom{A_k}{B_k}=0.
\end{equation}
The determinant is
\[
 72\left(1+\frac3{d+13}\right)\ne0,
\]
independent of $k$.  The induction starts with $A_1=B_1=0$, by rootlessness and Proposition~\ref{prop:no-two-thirds}; the zero vector also contributes zero.  Hence every $A_k$ and $B_k$ vanishes.

Every shell of $L$ is antipodal.  At a non-zero discriminant norm, the shell of $L^\#$ is the antipodal union of the corresponding $C_+$- and $C_-$-shells; its even harmonic sums vanish by the preceding argument.  The norms in $L$ and in its non-zero classes are incongruent modulo $2\Z$, so these are all the shells of $L^\#$.  Antipodality makes the odd harmonic sums vanish, proving the $5$-design assertion.
\end{proof}

\begin{corollary}[dual strong perfection]\label{cor:dual-strong-perfection}
Both $L$ and $L^\#$ are strongly perfect and hence perfect, eutactic, and extreme.  In particular, $L$ is dual strongly perfect.
\end{corollary}

\begin{proof}
A lattice is strongly perfect when its minimal shell is a spherical $4$-design, equivalently a $5$-design because the shell is antipodal.  The theorem applies to both minimal shells, and Venkov's criterion gives perfection and eutaxy, hence a strict local maximum of the Hermite invariant \cite{VenkovSurvey}.
\end{proof}

Strong perfection of the rank-$26$ lattice and its dual already appears in Venkov's treatment of $E_6$-sections of rank-$32$ even unimodular lattices \cite[Sec.~17]{VenkovSurvey}.  The argument above obtains the all-shell statement directly from the same gluing used for the $819$-point count.  It is the antipodal set $X\sqcup(-X)$, not the half-shell $X$ itself, that is a spherical $5$-design.  The non-zero cubic moment of $X$ will instead recover the Albert product in Section~\ref{sec:cayley-design}.

The scalar and quadratic identities also determine the complete theta characters.  With exponent equal to half the squared norm, write
\[
 f=\theta_L,\quad g=\theta_{C_+}=\theta_{C_-},\quad
 a=\theta_{E_6},\quad b=\theta_Y,\quad D=q\frac{d}{dq},
\]
where $Y$ is one non-zero $E_6$-coset.  The scalar theta series of $N$ and its weighted series for $Q(u,v)=3u^2-13v^2$ give
\begin{align}
 fa+2gb&=E_4^4-888\Delta E_4,\label{eq:theta-character-scalar}\\
 6aDf-26fDa+12bDg-52gDb&=-1872\Delta E_6.
 \label{eq:theta-character-quadratic}
\end{align}
Here the $E_4,E_6$ on the right are normalized Eisenstein series.  At $q^k$, after the lower coefficients have been determined, the new unknowns are $[q^k]f$ and $[q^{k-2/3}]g$.  Their coefficient matrix is
\[
 \begin{pmatrix}1&54\\6k&324k-1152\end{pmatrix},
 \qquad\det=-1152.
\]
Starting with $f(0)=1$, this determines both series recursively, without lattice uniqueness.  The first terms are
\begin{align}
 \theta_L(q)&=1+117936q^2+15235584q^3
                    +480957750q^4+\cdots,\label{eq:theta-L-expansion}\\
 \theta_{C_+}(q)&=819q^{4/3}+748800q^{7/3}
                    +53922960q^{10/3}+\cdots.
 \label{eq:theta-coset-expansion}
\end{align}
The required $E_6$ series are themselves elementary: the standard index-three gluing of $A_2^3$ gives $a=\alpha^3+2\beta^3$ and $b=3\alpha\beta^2$, where $\alpha=\theta_{A_2}$ and $\beta$ is the theta series of either non-zero $A_2$-coset \cite[Ch.~4]{ConwaySloane}.

\begin{remark}\label{rem:even-suffices}
The rank-$32$ gluing proves the theta count and the even spherical $4$-design identities without using a concrete model of $L$.  The association-scheme extraction below also needs the mixed cubic sum
\[
        \sum_{z\in X}\ip{x}{z}^2z=16x,
\]
but this identity is not an additional analytic input.  It follows from \eqref{eq:BBV4tensor} and the three-distance property in Lemma~\ref{lem:cubic-reduction}.
\end{remark}

We now derive the three-distance property directly from the lattice inequalities.

\begin{lemma}[the three possible inner products]\label{lem:three-inner-products}
Let $x,y\in X$, $x\ne y$.  Then
\[
        \ip{x}{y}\in
        \left\{-\frac43,\frac23,-\frac13\right\}.
\]
\end{lemma}

\begin{proof}
Since $x$ and $y$ lie in the same discriminant class, $x-y\in L$.  Hence
\[
        (x-y)^2=\frac{16}{3}-2\ip{x}{y}\ge4.
\]
Also $x+y\in C_-$, so Proposition~\ref{prop:BBV-package} gives
\[
        (x+y)^2=\frac{16}{3}+2\ip{x}{y}\ge\frac83.
\]
Thus
\[
        -\frac43\le \ip{x}{y}\le\frac23.
\]
Finally, because $(x-y)^2\in2\Z$, one has $\ip{x}{y}\equiv2/3\pmod\Z$.  The three displayed values are the only possibilities.
\end{proof}

Define relations on $X$ by
\begin{align*}
        R_0&=\{(x,x):x\in X\},\\
        R_1&=\{(x,y):x\ne y,\ \ip{x}{y}=-4/3\},\\
        R_2&=\{(x,y):x\ne y,\ \ip{x}{y}=2/3\},\\
        R_3&=\{(x,y):x\ne y,\ \ip{x}{y}=-1/3\}.
\end{align*}
Let $A_i$ be the adjacency matrix of $R_i$, and put $A_0=I$.

\begin{proposition}[valencies]\label{prop:valencies}
For every $x\in X$,
\[
        |R_1(x)|=18,
        \qquad |R_2(x)|=288,
        \qquad |R_3(x)|=512.
\]
Consequently
\begin{equation}\label{eq:sum-zero}
        \sum_{x\in X}x=0.
\end{equation}
\end{proposition}

\begin{proof}
Fix $x\in X$ and write
\[
        a=|R_1(x)|,
        \qquad b=|R_2(x)|,
        \qquad d=|R_3(x)|.
\]
Then
\begin{equation}\label{eq:valencies-0}
        a+b+d=818.
\end{equation}
The second moment gives
\[
        \sum_{y\in X}\ip{x}{y}^2=84x^2=224,
\]
so
\begin{equation}\label{eq:valencies-2}
        \frac{64}{9}+\frac{16}{9}a+\frac49b+\frac19d=224.
\end{equation}
The fourth moment gives
\[
        \sum_{y\in X}\ip{x}{y}^4
        =8\cdot3(x^2)^2=\frac{512}{3},
\]
so
\begin{equation}\label{eq:valencies-4}
        \frac{4096}{81}+\frac{256}{81}a+\frac{16}{81}b+\frac1{81}d=\frac{512}{3}.
\end{equation}
Solving \eqref{eq:valencies-0}--\eqref{eq:valencies-4} gives
\[
        a=18,
        \qquad b=288,
        \qquad d=512.
\]
Now let $s=\sum_{y\in X}y$.  For every $x\in X$,
\[
        \ip{x}{s}
        =\frac83+18\left(-\frac43\right)
             +288\left(\frac23\right)
             +512\left(-\frac13\right)=0.
\]
Using the second moment with $a=b=s$ gives
\[
        84(s,s)=\sum_{x\in X}\ip{s}{x}^2=0,
\]
and hence $s=0$.
\end{proof}

\begin{lemma}[cubic reduction from the even fourth moment]\label{lem:cubic-reduction}
For all $x,y\in X$,
\begin{equation}\label{eq:cubic-reduction}
        \sum_{z\in X}\ip{x}{z}^2\ip{y}{z}=16\ip{x}{y}.
\end{equation}
Equivalently,
\begin{equation}\label{eq:cubic-vector}
        \sum_{z\in X}\ip{x}{z}^2z=16x.
\end{equation}
\end{lemma}

\begin{proof}
Put
\[
 q(T)=(T+4/3)(T-2/3)(T+1/3)
     =T^3+T^2-\frac23T-\frac8{27}.
\]
For fixed $x\in X$, the three-distance property and $q(8/3)=24$ give
\[
 \sum_{z\in X}q((x,z))z=24x.
\]
The fourth and second moments give, respectively,
\[
 \sum_{z\in X}(x,z)^3z=24(x,x)x=64x,
 \qquad \sum_{z\in X}(x,z)z=84x.
\]
Using $\sum_z z=0$ and expanding $q$ therefore gives
\[
 \sum_{z\in X}(x,z)^2z
   =24x-64x+\frac23\cdot84x=16x.
\]
Pairing with $y$ proves the scalar identity, including the case $x=y$.
\end{proof}


\section{The 819 vectors form a generalized hexagon}
\label{sec:819-hexagon}

Let
\begin{equation}\label{eq:G-def}
        G=\bigl(\ip{x}{y}\bigr)_{x,y\in X}
        =\frac83A_0-\frac43A_1+\frac23A_2-\frac13A_3.
\end{equation}
The second moment identity immediately gives a rank-$26$ idempotent.

\begin{lemma}[the rank-$26$ idempotent]\label{lem:E26}
One has
\[
        G^2=84G.
\]
Consequently
\begin{equation}\label{eq:E26}
        F_3:=\frac1{84}G
        =\frac{2}{63}A_0-\frac1{63}A_1+\frac1{126}A_2-\frac1{252}A_3
\end{equation}
is an orthogonal idempotent of rank $26$.
\end{lemma}

\begin{proof}
The $(x,y)$ entry of $G^2$ is
\[
        \sum_{z\in X}\ip{x}{z}\ip{z}{y}=84\ip{x}{y},
\]
by \eqref{eq:BBV2}.  The trace of $G/84$ is
\[
        \frac1{84}\cdot819\cdot\frac83=26.
\]
\end{proof}

For $x,y\in X$ with $x\ne y$ and $\ip{x}{y}=r_k$, where
\[
        r_1=-\frac43,
        \qquad r_2=\frac23,
        \qquad r_3=-\frac13,
\]
define
\[
        p_{ij}^k(x,y)
        =\#\{z\in X:\ip{x}{z}=r_i,\ \ip{y}{z}=r_j\},
        \qquad 1\le i,j\le3.
\]

\begin{proposition}[intersection tables]\label{prop:intersection-tables}
The numbers $p_{ij}^k(x,y)$ depend only on $i,j,k$, not on the chosen pair $(x,y)$.  The non-trivial intersection tables are as follows; the table labelled $p^k$ gives $p_{ij}^k$ with rows indexed by $i$ and columns by $j$:
\[
\begin{array}{c|ccc}
 p^1_{ij} & 1&2&3\\ \hline
 1&1&16&0\\
 2&16&16&256\\
 3&0&256&256
\end{array}
\qquad
\begin{array}{c|ccc}
 p^2_{ij} & 1&2&3\\ \hline
 1&1&1&16\\
 2&1&142&144\\
 3&16&144&352
\end{array}
\qquad
\begin{array}{c|ccc}
 p^3_{ij} & 1&2&3\\ \hline
 1&0&9&9\\
 2&9&81&198\\
 3&9&198&304
\end{array}.
\]
Together with the valencies $k_1=18$, $k_2=288$, $k_3=512$, these are the intersection numbers of a symmetric three-class association scheme.
\end{proposition}

\begin{proof}
Fix $x,y\in X$ with $x\ne y$ and $\ip{x}{y}=t=r_k$.  Put
\[
        M_{ab}(t)=\sum_{z\in X}\ip{x}{z}^a\ip{y}{z}^b.
\]
For $0\le a,b\le2$ the preceding moment identities give
\begin{align*}
        M_{00}&=819,                 &
        M_{10}&=M_{01}=0,\\
        M_{20}&=M_{02}=224,          &
        M_{11}&=84t,\\
        M_{21}&=M_{12}=16t,          &
        M_{22}&=8\left(\frac{64}{9}+2t^2\right).
\end{align*}
Here $M_{10}=M_{01}=0$ follows from \eqref{eq:sum-zero}; $M_{21}=M_{12}=16t$ is Lemma~\ref{lem:cubic-reduction}; and $M_{22}$ is \eqref{eq:BBV4}.

The terms $z=x$ and $z=y$ contribute respectively $(8/3,t)$ and $(t,8/3)$ to the pair of inner products.  Hence the unknown $3\times3$ matrix $P^k=(p_{ij}^k(x,y))$ satisfies, for $0\le a,b\le2$,
\begin{equation}\label{eq:vandermonde-system}
        \sum_{i,j=1}^3p_{ij}^k(x,y)r_i^ar_j^b
        =M_{ab}(t)-\left(\frac83\right)^at^b-t^a\left(\frac83\right)^b.
\end{equation}
The coefficient matrix is the Vandermonde matrix
\[
        V=\begin{pmatrix}
        1&1&1\\
        -4/3&2/3&-1/3\\
        16/9&4/9&1/9
        \end{pmatrix},
        \qquad \det V=-2\ne0.
\]
Therefore \eqref{eq:vandermonde-system} has a unique solution, and that solution depends only on $t=r_k$.  Solving gives exactly the three displayed tables.  Thus the intersection numbers are independent of the chosen pair $(x,y)$, which is precisely the association-scheme condition.
\end{proof}

\begin{theorem}[the Bose--Mesner algebra]\label{thm:BM}
The matrices $A_0,A_1,A_2,A_3$ span the adjacency algebra of a symmetric three-class association scheme.  With primitive idempotents ordered by ranks
\[
        1,\qquad 324,\qquad 468,\qquad 26,
\]
the first eigenmatrix is
\begin{equation}\label{eq:P-matrix}
\begin{array}{c|rrrr|r}
        &A_0&A_1&A_2&A_3&\rank\\ \hline
F_0&1&18&288&512&1\\
F_1&1&5&2&-8&324\\
F_2&1&-3&-6&8&468\\
F_3&1&-9&72&-64&26.
\end{array}
\end{equation}
\end{theorem}

\begin{proof}
The intersection tables give constants $p_{ij}^k$ such that
\[
        A_iA_j=\sum_{k=0}^3p_{ij}^kA_k,
\]
with $p_{ii}^0=k_i$ and $p_{ij}^0=0$ for $i\ne j$.  Hence the adjacency matrices form a Bose--Mesner algebra.  The eigenmatrix is obtained by diagonalizing the intersection matrix for multiplication by $A_1$.  Equivalently, one may use the distance-polynomial recurrence encoded in the multiplication identities displayed below.  The final row is fixed by the rank-$26$ Gram idempotent \eqref{eq:E26}; the multiplicities are then $1,324,468,26$, summing to $819$.
\end{proof}

\begin{corollary}[the $P$-polynomial structure]\label{cor:intersection-array}
The graph $(X,R_1)$ is distance-regular of diameter $3$ with intersection array
\[
        \{18,16,16;1,1,9\}.
\]
Equivalently,
\begin{align}
        A_1^2&=18A_0+A_1+A_2,                         \label{eq:A1sq}\\
        A_1A_2&=16A_1+A_2+9A_3,                       \label{eq:A1A2}\\
        A_1A_3&=16A_2+9A_3.                           \label{eq:A1A3}
\end{align}
\end{corollary}

\begin{proof}
The identities are read directly from the intersection tables.  They show that $R_1,R_2,R_3$ are the distance-$1$, distance-$2$, and distance-$3$ relations of the graph $(X,R_1)$.  The displayed intersection array is then immediate: from $A_1^2$ one reads $b_0=18$, $a_1=1$, $c_2=1$; from $A_1A_2$ one reads $b_1=16$, $a_2=1$, $c_3=9$; and from $A_1A_3$ one reads $b_2=16$, $a_3=9$.  The value $c_1=1$ is automatic.
\end{proof}

\subsection{Zero-sum triples and the generalized hexagon}
\label{sec:hexagon}

We prove the passage between the scheme and the incidence geometry once, in a form that applies both to the lattice shell and to the Cayley-plane configuration.

\begin{proposition}[the association scheme and the hexagon determine each other]\label{prop:scheme-hexagon-equivalence}
Let $\mathfrak X=(X,R_0,R_1,R_2,R_3)$ be a symmetric three-class association scheme with first eigenmatrix \eqref{eq:P-matrix}.  For every $R_1$-edge $\{x,y\}$, let $z$ be the unique common $R_1$-neighbour of $x$ and $y$, and put
\[
        \ell(x,y)=\{x,y,z\}.
\]
The distinct triples $\ell(x,y)$ form the lines of a generalized hexagon of order $(2,8)$, whose point graph is $(X,R_1)$.  Conversely, if $H$ is a generalized hexagon of order $(2,8)$, then the distance relations in its point graph form an association scheme with first eigenmatrix \eqref{eq:P-matrix}.  These two constructions are inverse to one another.
\end{proposition}

\begin{proof}
For the first direction, the eigenmatrix determines the intersection numbers, giving the tables of Proposition~\ref{prop:intersection-tables} and the identities of Corollary~\ref{cor:intersection-array}.  In particular,
\[
        A_1^2=18A_0+A_1+A_2,
\]
so $p_{11}^{1}=1$.  Hence every $R_1$-edge lies in a unique triangle, and no clique in $(X,R_1)$ has more than three vertices.  The triples $\ell(x,y)$ therefore form a partial linear space.  Since $k_1=18$, every point lies on $18/2=9$ such triples.

The primitive idempotent $F_3$ gives a Euclidean representation $x\mapsto v_x$ with Gram matrix $84F_3$; equivalently
\[
        (v_x,v_x)=\frac83,
        \qquad
        (v_x,v_y)=
        \begin{cases}
        -4/3,&(x,y)\in R_1,\\
        2/3,&(x,y)\in R_2,\\
        -1/3,&(x,y)\in R_3.
        \end{cases}
\]
If $\ell=\{x,y,z\}$ is one of the above triangles, then the three off-diagonal inner products are all $-4/3$, and hence
\[
        (v_x+v_y+v_z)^2
        =3\cdot\frac83+6\left(-\frac43\right)=0.
\]
Thus $v_x+v_y+v_z=0$.  Consequently, for any point $p$ and any line $\ell=\{x,y,z\}$,
\[
        (v_p,v_x)+(v_p,v_y)+(v_p,v_z)=0.
\]
If $p\notin\ell$, the only triples of values from $\{-4/3,2/3,-1/3\}$ with sum zero are
\[
        \left(-\frac43,\frac23,\frac23\right),
        \qquad
        \left(\frac23,-\frac13,-\frac13\right).
\]
Thus the graph distances from $p$ to the three points of $\ell$ are respectively of type $(1,2,2)$ or $(2,3,3)$; if $p\in\ell$ they are of type $(0,1,1)$.  Hence every point has a unique nearest point on every line.

The graph $(X,R_1)$ has diameter $3$.  Point-point distances in the incidence graph are therefore at most $6$; point-line distances are at most $5$ by the preceding paragraph; and line-line distances are at most $6$ by applying the same paragraph to a point of the first line and the second line.  Thus the incidence graph has diameter at most $6$, and it has diameter exactly $6$ because $R_3$ is non-empty.

Point-point and point-line geodesics of incidence length $<6$ are unique.  For point-point distance $2$ this is the uniqueness of the line through an edge, and for point-point distance $4$ it is $c_2=1$.  For point-line distances $1$, $3$, and $5$, uniqueness follows from the unique nearest point on the line, together with $c_2=1$ in the distance-$5$ case.

If a shortest cycle had length $2m<12$, its two half-arcs would be distinct geodesics of length $m$.  Choose opposite vertices of point-point type when $m$ is even, and of point-line type when $m$ is odd.  This contradicts the uniqueness just proved.  Finally, if $x$ and $y$ are in relation $R_3$, then $c_3=9$ gives at least two graph geodesics of length $3$ from $x$ to $y$; the corresponding incidence geodesics have length $6$, and their union contains a cycle of length at most $12$.  The lower bound forces a $12$-cycle.  The incidence graph therefore has diameter $6$ and girth $12$, while the lines have $3=2+1$ points and each point lies on $9=8+1$ lines.  This is a generalized hexagon of order $(2,8)$, and its point graph is by construction $(X,R_1)$.  Since $R_i$ is the distance-$i$ relation of this graph, the association scheme is recovered from the hexagon.

Conversely, let $H$ be a generalized hexagon of order $(2,8)$.  Its incidence graph has diameter $6$ and girth $12$, so its point graph has diameter $3$.  Fix a base point $p$.  The point $p$ has $2(8+1)=18$ neighbours.  If $y$ is at point-graph distance $1$ from $p$, then the line $py$ contains exactly one further point, also at distance $1$ from $p$; hence $c_1=1$ and $a_1=1$.  If $y$ is at point-graph distance $2$, the incidence geodesic from $p$ to $y$ is unique, so $y$ has one predecessor at distance $1$; the line through this predecessor and $y$ contains exactly one further point at distance $2$, giving $c_2=1$ and $a_2=1$.  In these two cases there can be no further same-layer neighbour on another line through $y$: it would produce an incidence cycle of length $6$ or at most $10$, respectively, contrary to girth $12$.  If $y$ is at point-graph distance $3$, each of the nine lines through $y$ contains one point at distance $2$ from $p$ and one point at distance $3$ from $p$, so $c_3=9$ and $a_3=9$.  Therefore the intersection array of the point graph is
\[
        \{18,16,16;1,1,9\}.
\]
The distance relations form the corresponding $P$-polynomial three-class association scheme.  Diagonalizing the distance-polynomial recurrence gives exactly the eigenmatrix \eqref{eq:P-matrix}.  The lines of $H$ are precisely the triangles in the point graph: a triangle not contained in one line would give a $6$-cycle in the incidence graph.  Thus the two constructions recover one another.
\end{proof}

\begin{corollary}[the lattice hexagon]\label{cor:lattice-hexagon}
For $X=(C_+)_{8/3}$, the triples
\begin{equation}\label{eq:line-def}
 \{x,y,z\}\subset X,\qquad x+y+z=0,
\end{equation}
are the lines of a generalized hexagon of order $(2,8)$.
\end{corollary}

\begin{proof}
The preceding proposition applies to the scheme of Theorem~\ref{thm:BM}.  A triangle has zero sum by its Gram matrix.  Conversely, if three vectors in $X$ sum to zero, their three pairwise inner products sum to $-4$.  Each is at least $-4/3$, so all three equal $-4/3$ and the triple is a triangle.  Equivalently, for an adjacent pair $x,y$, the unique third point is $-x-y$: it lies in $C_+$ and has norm $8/3$.
\end{proof}

No Jordan product enters this construction.  The zero-sum triples become Jordan frames only in the Cayley-plane realization of Section~\ref{sec:cayley-design}.

\section{A generalized hexagon generates the lattice}
\label{sec:converse}

Let $\mathfrak X=(X,R_0,R_1,R_2,R_3)$ be the distance scheme of a generalized hexagon of order $(2,8)$.  By Proposition~\ref{prop:scheme-hexagon-equivalence}, its first eigenmatrix is \eqref{eq:P-matrix}.  Let $E=F_3$ be its rank-$26$ primitive idempotent and put
\begin{equation}\label{eq:E-converse}
 E=\frac{2}{63}\left(I-\frac12A_1+\frac14A_2-\frac18A_3\right),
 \qquad v_x=\sqrt{84}\,Ee_x\quad(x\in X).
\end{equation}
Here $e_x$ denotes the standard coordinate vector in $\R^X$.  The vectors $v_x$ span $V=E\R^X$ and have Gram matrix $84E$, with entries
\begin{equation}\label{eq:intro-Gram}
 (v_x,v_y)=
 \begin{cases}
 8/3,&x=y,\\
 -4/3,&(x,y)\in R_1,\\
 2/3,&(x,y)\in R_2,\\
 -1/3,&(x,y)\in R_3.
 \end{cases}
\end{equation}
In particular, $v_x+v_y+v_z=0$ on every line $\{x,y,z\}$.

\subsection{The index-three lattice and its moments}

Define
\begin{equation}\label{eq:Lambda-converse}
 \Lambda=\langle v_x:x\in X\rangle_\Z,
 \qquad
 B=252E=8I-4A_1+2A_2-A_3=9I-3A_1+3A_2-J.
\end{equation}
Here $J=\one\one^t$ is the all-one matrix.  The rationality of $E$ shows that $\Lambda$ is a lattice of rank $26$.

\begin{proposition}[the index-three even lattice]\label{prop:index-three-lattice}
The map
\[
 \eps:\Lambda\longrightarrow\F_3,
 \qquad \eps\left(\sum_x a_xv_x\right)=\sum_xa_x\pmod3
\]
is well-defined and surjective.  Its kernel $L_{\mathfrak X}$ is even and integral, and
\begin{equation}\label{eq:index3-converse}
 [\Lambda:L_{\mathfrak X}]=3,
 \qquad L_{\mathfrak X}\subseteq\Lambda^\#.
\end{equation}
It is generated by the vectors $v_x-v_y$ and $3v_x$.
\end{proposition}

\begin{proof}
Since $B\equiv2J\pmod3$, the relation $\sum_xa_xv_x=0$ implies $Ba=0$ and hence $\sum_xa_x\equiv0\pmod3$.  This proves well-definedness; surjectivity follows from $\eps(v_x)=1$.  If $\sum_xa_x\equiv0\pmod3$, then $a^tBb$ is divisible by $3$ for every integral $b$.  Thus
\[
 \left(\sum_xa_xv_x,\sum_yb_yv_y\right)=\frac13a^tBb\in\Z.
\]
Moreover, $a^tBa$ is even because $B$ is integral with even diagonal.  Its quotient by $3$ is therefore even.  The generator description is the usual generating set for the kernel of augmentation modulo $3$.
\end{proof}

We shall use moment identities that follow from the eigenmatrix alone.  For $a\in V$, write
\[
 S_j(a)=\sum_{x\in X}(a,v_x)^j,
 \qquad q(a)=(a,a).
\]

\begin{lemma}[scheme moments]\label{lem:scheme-moments}
For every $a\in V$,
\begin{equation}\label{eq:scheme-moments}
 S_1(a)=0,\qquad S_2(a)=84q(a),\qquad
 S_4(a)=24q(a)^2,\qquad
 S_5(a)=\frac56q(a)S_3(a).
\end{equation}
If $s=((a,v_x)^2)_{x\in X}$, then
\begin{equation}\label{eq:Bs}
 Bs=90s-18A_1s+24q(a)\one.
\end{equation}
\end{lemma}

\begin{proof}
The first two identities follow from $E\one=0$ and $E^2=E$.  For the higher moments normalize $u_x=\sqrt{3/8}\,v_x$ to the unit sphere in dimension $26$.  The inner products in each row, with their multiplicities, are
\[
 1\ (1),\qquad -\tfrac12\ (18),\qquad
 \tfrac14\ (288),\qquad -\tfrac18\ (512).
\]
The Gegenbauer polynomials of degrees $4$ and $5$ for this sphere are
\begin{align*}
 C_4^{12}(t)&=21840t^4-4368t^2+78,\\
 C_5^{12}(t)&=139776t^5-43680t^3+2184t.
\end{align*}
Each has row sum zero on the displayed distribution.  The addition formula for spherical harmonics therefore gives
\[
 \sum_x C_j^{12}((b,u_x))=0\qquad(j=4,5)
\]
for every unit vector $b$; see \cite{DGS}.  Expanding these identities and using the first two moments gives the last two identities in \eqref{eq:scheme-moments}.

For \eqref{eq:Bs}, the relevant Krein vanishing is forced by dimension.  The matrix $E\circ E$ is the Gram matrix of the symmetric tensors $(Ee_x)\otimes(Ee_x)$.  Hence
\[
 \rank(E\circ E)\leq\dim\operatorname{Sym}^2(V)
       =\binom{27}{2}=351.
\]
Since $E\circ E$ belongs to the Bose--Mesner algebra, a non-zero $F_2$-coefficient would give rank at least $\rank F_2=468$.  Thus that coefficient is zero, equivalently $q_{33}^{\,2}=0$.  The eigenmatrix gives the remaining coefficients:
\begin{equation}\label{eq:schur-square}
 E\circ E=\frac{2}{63}F_0+\frac1{441}F_1+\frac1{441}F_3.
\end{equation}
All three coefficients are positive and $1+324+26=351$, so equality holds in the rank bound: the symmetric tensors span $\operatorname{Sym}^2(V)$.

The column space of $E\circ E$ is spanned by entrywise products of vectors in $E\R^X$: if the columns of $U$ are an orthonormal basis of $E\R^X$, then $E\circ E$ is the sum of the outer products of the vectors $U_i\circ U_j$.  Thus the vector $s$ lies in this column space, and $F_2s=0$.  Also $F_0s=(84q(a)/819)\one$.  The eigenvalues of $A_2+5A_1-27I$ on $F_0,F_1,F_2,F_3$ are $351,0,-48,0$, respectively.  Consequently
\[
 A_2s=27s-5A_1s+36q(a)\one.
\]
Substitution in the last expression for $B$ in \eqref{eq:Lambda-converse}, together with $Js=84q(a)\one$, proves \eqref{eq:Bs}.
\end{proof}

\subsection{Saturation}

\begin{theorem}[saturation]\label{thm:saturation}
The index-three lattice satisfies
\[
 L_{\mathfrak X}=\Lambda^\#.
\]
\end{theorem}

\begin{proof}
Abbreviate $L=L_{\mathfrak X}$ and $M=\Lambda^\#$.  We already know $L\subseteq M$ and $[\Lambda:L]=3$.  We shall prove that every vector of $M$ has even norm.  This makes $M$ integral, whence
\[
 L\subseteq M\subseteq M^\#=\Lambda.
\]
The index-three condition then gives $M=L$ or $M=\Lambda$.  The latter is impossible because $v_x^2=8/3$.

Since $L$ is even and $(M,L)\subseteq\Z$, norms are constant modulo $2\Z$ on each coset of $M/L$.  Suppose some coset has non-even norm, and choose a shortest vector $\alpha$ in it.  Put
\[
 m=(\alpha,\alpha),\qquad k_x=(\alpha,v_x)\in\Z,
 \qquad S_j=\sum_x k_x^j.
\]
If $(x,y)\in R_2$, then $v_x-v_y\in L$ has norm $4$.  Minimality of $\alpha$ under addition and subtraction of this vector gives
\[
 |k_x-k_y|\le2.
\]
The $R_2$ graph has diameter $2$: its common-neighbour numbers for distinct vertices are $p_{22}^1=16$, $p_{22}^2=142$, and $p_{22}^3=81$.  Thus any two labels differ by at most $4$.  The labels on every line sum to zero.  If the maximum were at least $3$, every label would be at least $-1$, and a line through a maximum would have positive sum.  Applying the same argument to the negative labels proves
\begin{equation}\label{eq:five-labels}
 k_x\in\{-2,-1,0,1,2\}.
\end{equation}

The second and fourth moments give $S_2=84m$ and $S_4=24m^2$, so $S_2^2=294S_4$.  Since $S_2,S_4$ are integers and $294=2\cdot3\cdot7^2$, it follows that $42\mid S_2$, or $m\in\tfrac12\Z$.  On the five labels in \eqref{eq:five-labels},
\[
 k^5=5k^3-4k.
\]
Together with $S_1=0$ and the fifth moment this gives $(m-6)S_3=0$.  Our coset has non-even norm, so $m\ne6$ and $S_3=0$.  The equations $S_1=S_3=0$ show that each non-zero label occurs as often as its negative.  In particular,
\begin{equation}\label{eq:label-two-count}
 \#\{x:|k_x|=2\}=\frac{S_4-S_2}{12}=m(2m-7)
\end{equation}
is even.  This excludes every odd integral value of $m$.

To exclude the remaining half-integral values, set
\[
 s=(k_x^2)_x,\qquad h=(k_x^4)_x,\qquad
 g=h-\frac m2s.
\]
Differentiating the fifth-moment identity in \eqref{eq:scheme-moments} at $\alpha$ gives
\[
 \sum_x k_x^4v_x
   =\frac13\alpha S_3+\frac m2\sum_x k_x^2v_x.
\]
Since $S_3=0$, this says $Bg=0$.

For a line $\ell=\{x,y,z\}$, write its labels as $a,b,c$, so $a+b+c=0$.  The elementary identities
\begin{align*}
 \sum_{i\ne j}a_i^2a_j^2&=a^4+b^4+c^4,\\
 \sum_{i\ne j}a_i^2a_j^4&=a^6+b^6+c^6-6(abc)^2
\end{align*}
hold with $(a_1,a_2,a_3)=(a,b,c)$.  Define the integer
\[
 U=\sum_{\ell=\{x,y,z\}}(k_xk_yk_z)^2.
\]
Each point is on nine lines, and each ordered adjacent pair lies on one line.  Summing the two identities therefore gives
\[
 s^tA_1s=9S_4,\qquad s^tA_1h=9S_6-6U.
\]
Contracting \eqref{eq:Bs} with $g$ now yields
\begin{align*}
 0=(Bs)^tg
 &=90\left(S_6-\frac m2S_4\right)
   -18\left(9S_6-6U-\frac{9m}{2}S_4\right)
   +24m\left(S_4-\frac m2S_2\right)\\
 &=-72S_6+432m^3+108U.
\end{align*}
Finally, $k^6=5k^4-4k^2$ on \eqref{eq:five-labels}, so $S_6=120m^2-336m$ and
\begin{equation}\label{eq:saturation-contraction}
 U+4m^3-80m^2+224m=0.
\end{equation}
If $m=r/2$ with $r$ odd, then $4m^3=r^3/2$ is not an integer, whereas $U$, $80m^2=20r^2$, and $224m=112r$ are integers.  This is a contradiction.  Thus every coset of $M/L$ has even norm, proving the theorem.
\end{proof}

\begin{corollary}[determinant and minimum]\label{cor:converse-det-min}
The lattice $L_{\mathfrak X}$ is even of rank $26$, determinant $3$, and minimum $4$.
\end{corollary}

\begin{proof}
The identities $L_{\mathfrak X}=\Lambda^\#$ and $[\Lambda:L_{\mathfrak X}]=3$ give $\det\Lambda=1/3$ and $\det L_{\mathfrak X}=3$.  A difference $v_x-v_y$ with $(x,y)\in R_2$ has norm $4$.  If $r\in L_{\mathfrak X}$ had norm $2$, its integral labels $z_x=(r,v_x)$ would satisfy
\[
 \sum_xz_x^2=168,\qquad \sum_xz_x^4=96
\]
by Lemma~\ref{lem:scheme-moments}.  This contradicts $z_x^4\ge z_x^2$ for integer $z_x$.
\end{proof}

\subsection{Equivalence, uniqueness, and automorphisms}
\label{sec:equivalence}

We can now compare the two constructions, not merely their numerical parameters.  Starting with a lattice, saturation shows that the shortest shell generates its full dual.  Starting with a hexagon, the theta count shows that the reconstructed shell contains no additional points.  We first make these inverse statements explicit, and then record their consequences for uniqueness and automorphisms.

\begin{proposition}[the two constructions are inverse]\label{prop:inverse}
The discriminant-shell construction and the idempotent construction give inverse bijections between isometry classes of even rank-$26$ lattices of determinant $3$ and minimum $4$, and isomorphism classes of generalized hexagons of order $(2,8)$.
\end{proposition}

\begin{proof}
Start with $L$ and $X=(C_+)_{8/3}$.  The lattice $\langle X\rangle_\Z\subseteq L^\#$ has the Gram matrix used in the converse construction.  By saturation it has determinant $1/3$, equal to that of $L^\#$; hence $\langle X\rangle_\Z=L^\#$.  Taking duals recovers $L$.  Choosing $C_-$ instead gives the isomorphic configuration $-X$.

Conversely, start with a hexagon and its scheme $\mathfrak X$.  Saturation gives $L_{\mathfrak X}=\Lambda^\#$.  The $819$ distinct vectors $v_x$ have norm $8/3$ and lie in one non-zero class of $\Lambda/L_{\mathfrak X}$.  Proposition~\ref{prop:BBV-package} says that this class contains exactly $819$ vectors of that norm.  Thus they form its entire shortest shell, recovering the scheme and, by Proposition~\ref{prop:scheme-hexagon-equivalence}, the hexagon.
\end{proof}

\begin{corollary}\label{cor:hexagon-uniqueness}
The generalized hexagon of order $(2,8)$ is unique.
\end{corollary}

\begin{proof}
Apply Proposition~\ref{prop:inverse} and the positive-definite lattice uniqueness theorem proved in Section~\ref{sec:niemeier-descent}, Theorem~\ref{thm:positive-definite-uniqueness}.  That proof uses the forward construction and the line-deletion correspondence, but not the present corollary.
\end{proof}

Let $\Aut^+(L)$ be the subgroup preserving $C_+$.  The equivalence also identifies the automorphism groups.

\begin{proposition}\label{prop:aut-identification-lattice}
If $H$ is the hexagon recovered from $L$, then
\[
 \Aut^+(L)\cong\Aut(H),\qquad
 \Aut(L)\cong C_2\times\Aut(H).
\]
\end{proposition}

\begin{proof}
Every element of $\Aut^+(L)$ permutes $X$, preserves its inner products, and therefore acts on the recovered hexagon.  This action is faithful because $X$ spans $L\otimes\R$.  Conversely, a hexagon automorphism preserves point-graph distances, hence the Gram matrix of $X$.  It consequently extends uniquely to an orthogonal transformation.  Proposition~\ref{prop:inverse} shows that $X$ generates $L^\#$, so the extension preserves $L^\#$ and its dual $L$, and belongs to $\Aut^+(L)$.  These constructions are inverse.

The full automorphism group acts on $L^\#/L\cong\Z/3$.  An automorphism either preserves the two non-zero classes or interchanges them.  The central isometry $-1$ interchanges them and does not belong to $\Aut^+(L)$.  Thus every element of $\Aut(L)$ is uniquely a product of an element of $\{\pm1\}$ and an element of $\Aut^+(L)$, proving the direct-product assertion.
\end{proof}

\begin{remark}[recognition and the order calculation]\label{rem:deferred-recognition}
Comparison with the standard Steinberg--Tits triality hexagon could already identify the groups in the preceding proposition.  We postpone that comparison to Section~\ref{sec:triality-identification}.  Before naming the group, Section~\ref{sec:niemeier-descent} will recover the lattice from its rank-$24$ deletion datum and calculate the order from the Niemeier picture.  This keeps the deletion-and-extension mechanism visible, as in the code construction, rather than replacing it by recognition of a known group.
\end{remark}

\section{The tight Cayley-plane design is geometrically unique}
\label{sec:cayley-design}

We now pass from the abstract hexagon to its realization in the Cayley plane.  The point requiring proof is that an isomorphism of generalized hexagons preserves the Gram matrix of the centered $819$-point configurations and hence extends to an element of $O(26)$, whereas geometric uniqueness asks for an element of the much smaller compact group $F_4(\mathbb R)$.  The missing tensor is the third moment.

Let
\[
        J=\operatorname{Herm}_3(\mathbb O)
\]
be the compact real Albert algebra, with identity $E$, Jordan product $\circ$, trace $\operatorname{Tr}$, and positive trace form
\[
        T(a,b)=\operatorname{Tr}(a\circ b).
\]
Thus $\operatorname{Tr}(E)=T(E,E)=3$.  The Cayley plane is the manifold
\[
        \mathbb{O}P^2
        =\{p\in J:p\circ p=p,\ \operatorname{Tr}(p)=1,\ p\text{ primitive}\}.
\]
For the metric induced by the trace form, the full isometry group of the Cayley plane is its Jordan automorphism group,
\begin{equation}\label{eq:cayley-isometry-group}
        \operatorname{Isom}(\mathbb{O}P^2,T)
        =\Aut(J)\cong F_4(\mathbb R),
\end{equation}
the compact real form of type $F_4$; see \cite[Chs.~5 and~7]{SpringerVeldkamp}.  The distinction used below is that an ambient element of $O(V,T)$ carrying one finite design to another is not yet known to preserve $\mathbb{O}P^2$ as a whole.
Write
\[
        V=J^0=\{a\in J:\operatorname{Tr}(a)=0\}=E^\perp
\]
and center the Cayley plane in the sphere of radius $\sqrt{8/3}$ in $V$ by
\begin{equation}\label{eq:cayley-centering}
        \widehat p=2p-\frac23E.
\end{equation}
For $p,q\in\mathbb{O}P^2$ one has
\begin{equation}\label{eq:cayley-angle-conversion}
        T(\widehat p,\widehat q)
        =4T(p,q)-\frac43,
        \qquad
        T(\widehat p,\widehat p)=\frac83.
\end{equation}
We call $T(p,q)$ the Jordan angle of $p$ and $q$.

We shall use the following standard consequence of the spectral theorem for Euclidean Jordan algebras; see, for example, \cite[Chs.~5 and~7]{SpringerVeldkamp}.

\begin{lemma}[low-degree invariants of the Albert module]\label{lem:albert-invariants}
On the trace-zero Albert module $V$, the invariant polynomial algebra is
\[
        \mathbb R[V]^{F_4(\mathbb R)}
        =\mathbb R[q,c],
        \qquad
        q(a)=T(a,a),
        \qquad
        c(a)=\operatorname{Tr}(a^3).
\]
In particular, the invariant quadratic, cubic, and quartic polynomials have dimensions $1$, $1$, and $1$, generated respectively by $q$, $c$, and $q^2$.
\end{lemma}

\begin{proof}
Every $a\in V$ is carried by an automorphism of $J$ to a spectral decomposition
\[
        a=\lambda_1e_1+\lambda_2e_2+\lambda_3e_3,
        \qquad
        \lambda_1+\lambda_2+\lambda_3=0,
\]
for a Jordan frame $(e_1,e_2,e_3)$.  The group $F_4(\mathbb R)$ is transitive on Jordan frames, and the stabilizer of a frame induces the full symmetric group on its three members.  Restriction to the diagonal plane is therefore injective on invariant polynomials, and its image consists of symmetric polynomials in $\lambda_1,\lambda_2,\lambda_3$ subject to $\lambda_1+\lambda_2+\lambda_3=0$.  The algebra of these symmetric polynomials is generated by
\[
        \lambda_1^2+\lambda_2^2+\lambda_3^2=q(a)
\]
and
\[
        \lambda_1^3+\lambda_2^3+\lambda_3^3=c(a).
\]
Thus the restriction of any invariant polynomial $f$ agrees with a polynomial $P(q,c)$.  Since $q$ and $c$ are invariant and every orbit meets the diagonal plane, $f=P(q,c)$ on all of $V$.  This proves the stated equality of invariant algebras.
\end{proof}

Fix the normalized $F_4(\mathbb R)$-invariant probability measure $\mu$ on $\mathbb{O}P^2$.  A finite set $\mathcal D\subset\mathbb{O}P^2$ is a \emph{projective $t$-design} if
\begin{equation}\label{eq:cayley-design-definition}
 \frac1{|\mathcal D|}\sum_{p\in\mathcal D}f(p)
       =\int_{\mathbb{O}P^2}f(p)\,d\mu(p)
\end{equation}
for every restriction to the Cayley plane of a polynomial $f$ on $J$ of degree at most $t$ \cite{Hoggar1982,Hoggar1989}.  The degree here is measured in the coordinates of the trace-one idempotent $p$.  Thus a projective design is an exact finite averaging rule for low-degree functions on the Cayley plane, with all points given equal weight.

For example, degree one gives
\[
 \frac1{|\mathcal D|}\sum_{p\in\mathcal D}p=\frac13E,
 \qquad \sum_{p\in\mathcal D}\widehat p=0.
\]
Products of two or three linear functions $T(a,\widehat p)$ give the second and third moments of the centered configuration.  The second moment recovers the trace form on $V$ and shows that the centered points span $V$.  The third moment remembers the Albert cubic and, by polarization, the traceless Jordan product.  These two roles will be made explicit in the next lemma.  In particular, projective exactness is not spherical exactness on the whole sphere in $V$: the invariant measure is supported on the Cayley plane, and its cubic moment need not vanish.

The adjective \emph{tight} means that the design attains the projective-design lower bound for its cardinality.  For projective $5$-designs in $\mathbb{O}P^2$ this bound is $819$; equality also forces the three Jordan angles recorded below in \eqref{eq:cayley-tight-data} \cite{Hoggar1982,Hoggar1989,Nasmith2022}.  The size and angle set therefore match the shortest discriminant shell of the rank-$26$ lattice.  The even moments will recover its association scheme, whereas the cubic moment will distinguish equivalence in the Cayley plane from mere orthogonal equivalence of the centered vectors.  For this latter step tightness is unnecessary, so we first work with an arbitrary projective $3$-design.

\begin{lemma}[the second and third design moments]\label{lem:cayley-cubic-moment}
Let $\mathcal D\subset\mathbb{O}P^2$ be a projective $3$-design with $N$ points.  For $a,b\in V$,
\begin{align}
 \sum_{p\in\mathcal D}T(a,\widehat p)T(b,\widehat p)
   &=\frac{4N}{39}T(a,b),\label{eq:general-second-moment}\\
 \sum_{p\in\mathcal D}T(a,\widehat p)^3
   &=\frac{8N}{273}\operatorname{Tr}(a^3).
   \label{eq:general-cubic-moment}
\end{align}
In particular, the centered design spans $V$.  For $N=819$, the coefficients are $84$ and $24$.
\end{lemma}

\begin{proof}
Cubature replaces the sums by $N$ times the invariant integrals.  By Lemma~\ref{lem:albert-invariants}, the quadratic integral is a multiple of $q(a)$.  Taking its trace over an orthonormal basis of $V$ gives
\[
 \int T(a,\widehat p)^2\,d\mu(p)
   =\frac{8/3}{26}q(a)=\frac4{39}q(a).
\]
Polarization proves \eqref{eq:general-second-moment}.

The cubic integral is a multiple of $\operatorname{Tr}(a^3)$.  To compute it independently of any finite configuration, fix $p_0\in\mathbb{O}P^2$.  For invariantly distributed $p$, the variable $t=T(p_0,p)$ has beta distribution with parameters $(4,8)$: its density is proportional to $t^3(1-t)^7$ on $[0,1]$.  This is the Jacobi weight for the Cayley plane \cite{Hoggar1982,Nasmith2022}.  Thus
\[
 \int t^j\,d\mu=\frac{(4)_j}{(12)_j},
 \qquad (a)_j=a(a+1)\cdots(a+j-1),
\]
and \eqref{eq:cayley-angle-conversion} gives
\[
 \int T(\widehat p_0,\widehat p)^3\,d\mu(p)
 =\int(4t-4/3)^3\,d\mu=\frac{128}{2457}.
\]
The eigenvalues of $\widehat p_0$ are $4/3,-2/3,-2/3$, so $\operatorname{Tr}(\widehat p_0^3)=16/9$.  The cubic coefficient is therefore $8/273$, proving \eqref{eq:general-cubic-moment}.
\end{proof}

Define the traceless Jordan product by
\[
 a*b=a\circ b-\frac13T(a,b)E\qquad(a,b\in V).
\]
Polarization of \eqref{eq:general-cubic-moment} gives the explicit reconstruction formula
\begin{equation}\label{eq:design-product}
 a*b=\frac{273}{8N}\sum_{p\in\mathcal D}
       T(a,\widehat p)T(b,\widehat p)\widehat p.
\end{equation}
Indeed, pairing both sides with $c\in V$ gives the polarized cubic, since $T(a*b,c)=T(a\circ b,c)$.  For $N=819$ the coefficient in \eqref{eq:design-product} is $1/24$.

\begin{lemma}[cubic rigidity for projective $3$-designs]\label{lem:cayley-cubic-rigidity}
Let $\mathcal D,\mathcal D'\subset\mathbb{O}P^2$ be projective $3$-designs.  Every bijection $\phi:\mathcal D\to\mathcal D'$ preserving Jordan angles extends uniquely to an element of $\Aut(J)\cong F_4(\mathbb R)$.
\end{lemma}

\begin{proof}
By \eqref{eq:cayley-angle-conversion}, the map $\widehat p\mapsto\widehat{\phi(p)}$ preserves the Gram matrix.  In particular, it preserves all linear relations among the centered points: the squared norm of any proposed relation is computed from that Gram matrix.  Since both centered designs span $V$ by Lemma~\ref{lem:cayley-cubic-moment}, the map extends uniquely to $g_0\in O(V,T)$.

Since $\phi$ is a bijection, the two designs have the same cardinality $N$.  Apply the product formula \eqref{eq:design-product} to both designs with this common value of $N$.  Orthogonality gives $T(a,\widehat p)=T(g_0a,g_0\widehat p)$, and $g_0$ permutes the centered points according to $\phi$.  Termwise application of $g_0$ to the sum therefore gives
\[
 g_0(a*b)=g_0(a)*g_0(b).
\]
Thus the extension preserves not only the metric but also the traceless Albert product.  Extend $g_0$ to $J=\R E\oplus V$ by fixing $E$.  The identity
\[
 (\alpha E+a)\circ(\beta E+b)
 =\left(\alpha\beta+\frac13T(a,b)\right)E
       +\alpha b+\beta a+a*b
\]
shows that the extension preserves the full Jordan product.  Hence it is an automorphism of $J$, and uniqueness follows from spanning.
\end{proof}

\subsection{Tight designs and the hexagon}

We now specialize the general moment argument to a tight projective $5$-design.  Tightness supplies the finite angle set needed to compare the configuration with the lattice shell.  The standard bound and annihilator polynomial give
\begin{equation}\label{eq:cayley-tight-data}
 |\mathcal D|=819,\qquad
 A(\mathcal D)=\left\{0,\frac14,\frac12\right\}.
\end{equation}
Here $A(\mathcal D)$ is the set of Jordan angles between distinct points.  The annihilator is, up to a non-zero scalar,
\begin{equation}\label{eq:cayley-annihilator}
 uP_2^{(8,4)}(2u-1)=15u(2u-1)(4u-1);
\end{equation}
see \cite{Hoggar1982,Hoggar1989,Nasmith2022}.

\begin{lemma}[the centered even moments]\label{lem:cayley-even-moments}
For a tight projective $5$-design $\mathcal D$, the centered vectors satisfy
\begin{align}
 \sum_{p\in\mathcal D}T(a,\widehat p)T(b,\widehat p)
 &=84T(a,b),\label{eq:cayley-second-moment}\\
 \sum_{p\in\mathcal D}T(a,\widehat p)T(b,\widehat p)
                       T(c,\widehat p)T(d,\widehat p)
 &=8\bigl(T(a,b)T(c,d)+T(a,c)T(b,d)+T(a,d)T(b,c)\bigr).
 \label{eq:cayley-fourth-moment}
\end{align}
\end{lemma}

\begin{proof}
The second moment is \eqref{eq:general-second-moment}.  By Lemma~\ref{lem:albert-invariants}, the invariant quartic integral has the form
\[
 I_4(a)=\int T(a,\widehat p)^4\,d\mu(p)=\gamma q(a)^2.
\]
Taking the Euclidean Laplacian on $V$ gives
\[
 \Delta I_4=12\cdot\frac83\cdot\frac4{39}q
          =\frac{128}{39}q,
 \qquad \Delta(q^2)=112q.
\]
Thus $\gamma=8/273$.  Multiplication by $819$ and polarization give \eqref{eq:cayley-fourth-moment}.
\end{proof}

\begin{proposition}[a tight design recovers the hexagon]\label{prop:cayley-design-hexagon}
Let $\mathcal D\subset\mathbb{O}P^2$ be a tight projective $5$-design.  Declare Jordan-orthogonal pairs adjacent and their three-point cliques to be lines.  The resulting geometry is a generalized hexagon of order $(2,8)$, whose point-graph distances $1,2,3$ correspond respectively to the Jordan angles $0,1/2,1/4$.
\end{proposition}

\begin{proof}
Equations \eqref{eq:cayley-angle-conversion} and \eqref{eq:cayley-tight-data} give the three centered inner products $-4/3,2/3,-1/3$.  Together with Lemma~\ref{lem:cayley-even-moments}, these are exactly the hypotheses used in the moment proofs of Proposition~\ref{prop:valencies}, Lemma~\ref{lem:cubic-reduction}, and Proposition~\ref{prop:intersection-tables}.  Thus the angle relations have first eigenmatrix \eqref{eq:P-matrix}.  Proposition~\ref{prop:scheme-hexagon-equivalence} now gives the hexagon and the stated correspondence of distances and angles.  Its lines are Jordan frames: their centered vectors sum to zero, so the three primitive idempotents sum to $E$.
\end{proof}

Fix the classical $819$-point tight design $\mathcal D_0$ constructed by Hoggar and realized in octonionic form by Nasmith \cite{Hoggar1984,Hoggar1989,Nasmith2022}.

\begin{theorem}[geometric uniqueness of the tight Cayley-plane design]\label{thm:cayley-design-uniqueness}
Every tight projective $5$-design in $\mathbb{O}P^2$ is conjugate under $F_4(\mathbb R)$ to $\mathcal D_0$.
\end{theorem}

\begin{proof}
Let $\mathcal D$ be a tight projective $5$-design.  Proposition~\ref{prop:cayley-design-hexagon} associates generalized hexagons of order $(2,8)$ to $\mathcal D$ and to $\mathcal D_0$.  Corollary~\ref{cor:hexagon-uniqueness} supplies an incidence isomorphism
\[
 \phi:\mathcal D\longrightarrow\mathcal D_0.
\]
It preserves distances in the point graphs.  By Proposition~\ref{prop:cayley-design-hexagon}, those distances determine all three Jordan angles, so $\phi$ preserves these angles as well.  Both designs are projective $3$-designs, and Lemma~\ref{lem:cayley-cubic-rigidity} therefore extends $\phi$ uniquely to an element of $F_4(\mathbb R)$.
\end{proof}

\begin{corollary}[the geometric stabilizer]\label{cor:cayley-design-stabilizer}
Let $H_0$ be the hexagon recovered from $\mathcal D_0$.  Restriction to the $819$ points induces an isomorphism
\[
 \operatorname{Stab}_{F_4(\mathbb R)}(\mathcal D_0)
          \xrightarrow{\ \sim\ }\Aut(H_0).
\]
\end{corollary}

\begin{proof}
An element of the geometric stabilizer preserves the trace form and hence Jordan orthogonality.  It therefore preserves the point-line incidence structure of $H_0$, giving the displayed restriction homomorphism.

This homomorphism is injective.  An element in its kernel fixes every point of $\mathcal D_0$, hence every centered point.  These span $V$ by Lemma~\ref{lem:cayley-cubic-moment}, so the element is the identity on $V$.  It also fixes the Jordan identity $E$, and is consequently the identity on $J$.

For surjectivity, let $\sigma\in\Aut(H_0)$.  It preserves point-graph distances, which determine the Jordan angles by Proposition~\ref{prop:cayley-design-hexagon}.  Lemma~\ref{lem:cayley-cubic-rigidity} extends $\sigma$ uniquely to an element of $F_4(\mathbb R)$, and that extension stabilizes $\mathcal D_0$.  Thus restriction is onto.  By Proposition~\ref{prop:aut-identification-lattice}, the same abstract group is $\Aut^+(L)$.
\end{proof}

The same moment formula recovers the Albert algebra from the oriented lattice itself.  This is a consequence of the uniqueness argument, not an input to the lattice or hexagon construction.

\begin{corollary}[the Albert algebra of an oriented lattice]\label{cor:oriented-albert}
Let $V=L\otimes\R$ and $X=(C_+)_{8/3}$.  Adjoin a vector $E\perp V$ with $E^2=3$ and define
\begin{equation}\label{eq:lattice-albert-product}
 a*_X b=\frac1{24}\sum_{x\in X}(a,x)(b,x)x
 \qquad(a,b\in V).
\end{equation}
On $J_L=\R E\oplus V$, the product
\begin{equation}\label{eq:lattice-full-jordan-product}
 (\alpha E+a)\circ(\beta E+b)
 =\left(\alpha\beta+\frac13(a,b)\right)E
       +\alpha b+\beta a+a*_Xb
\end{equation}
is the compact real Albert product, with identity $E$ and the given Euclidean form as trace form.  The vectors
\begin{equation}\label{eq:lattice-primitive-idempotents}
 p_x=\frac{x}{2}+\frac E3\qquad(x\in X)
\end{equation}
form a tight $819$-point projective $5$-design of primitive trace-one idempotents.  Every hexagon line gives a Jordan frame with sum $E$.
\end{corollary}

\begin{proof}
The hexagon of $X$ is isomorphic to that of the classical design $\mathcal D_0$, by Corollary~\ref{cor:hexagon-uniqueness}.  Such an isomorphism preserves point-graph distances and hence the centered Gram matrices.  It therefore extends to an isometry from $V$ to the trace-zero part of the classical Albert algebra.  Formula~\eqref{eq:design-product}, with $N=819$, identifies \eqref{eq:lattice-albert-product} with its traceless product.  Extending the isometry by the identity vector identifies \eqref{eq:lattice-full-jordan-product} with the full Albert product.  Thus no choice of a hexagon isomorphism remains in the displayed intrinsic formula.

For the normalization, Lemma~\ref{lem:cubic-reduction} gives $x*_Xx=2x/3$.  Substitution in \eqref{eq:lattice-full-jordan-product} shows directly that
\[
 p_x\circ p_x=p_x,\qquad (p_x,E)=1,\qquad
 (p_x,p_x)=\frac14\frac83+\frac19\,3=1.
\]
The isometry above identifies these vectors with the points of $\mathcal D_0$, so they are primitive and form the asserted design.  On a line $x+y+z=0$, their sum is $E$, and the pairwise Jordan products are zero; hence they form a Jordan frame.
\end{proof}

The chosen non-zero discriminant class is essential to the formula: replacing $X$ by $-X$ replaces $*_X$ by $-*_X$.  The two resulting unital Albert algebras are isomorphic by $\alpha E+a\mapsto\alpha E-a$.  Thus the oriented lattice determines a product, while the unoriented lattice determines its isomorphism class.

\begin{remark}[why hexagon uniqueness is not by itself enough]
The lattice--hexagon correspondence determines the centered configuration up to $O(26)$.  It does not by itself show that an orthogonal equivalence preserves the whole Cayley plane.  The third moment supplies the additional product tensor: formula~\eqref{eq:design-product} forces the equivalence to preserve the traceless Albert product, and fixing $E$ then recovers the full Jordan product.  This final upgrade to $F_4(\mathbb R)$ requires no classification of integral Albert algebras.
\end{remark}

\section{The Niemeier descent, uniqueness, and the group order}
\label{sec:niemeier-descent}

The rank-$24$ construction plays the same structural role as the odd-Golay descent in the code paper.  There an intrinsic flag is deleted and the smaller code, together with its deep-hole coset, reconstructs the original code and yields its group order.  Here we delete a hexagon line, retain an index-four sublattice of $N(A_1^{24})$ with a Golay trio and the gluing data, and reverse the construction.

Our aim is to prove uniqueness and calculate the automorphism-group order from these data before identifying the group.  The passage through a rank-$27$ odd unimodular lattice makes the deletion and its characteristic discriminant gluing explicit.  The essential additional step is a finite Golay calculation: root sign changes leave eight classes of admissible data, and explicit coordinate permutations join these classes into a single orbit.  No uniqueness theorem for the rank-$26$ or rank-$27$ lattice, or for the hexagon, is used in this section.

\subsection{The rank-\texorpdfstring{$27$}{27} extension and a primitive line}

Let $c$ be a vector of norm $3$, orthogonal to $L\otimes\R$.  The gluing
\begin{equation}\label{eq:rank27-glue}
 M=(L\oplus\Z c)
   \cup(C_+\oplus(c/3+\Z c))
   \cup(C_-\oplus(-c/3+\Z c))
\end{equation}
is an odd unimodular lattice of rank $27$.  Since $\min L=4$ and $\min C_\pm=8/3$, it has minimum $3$, and $c$ is characteristic.  Its norm-$3$ shell is
\begin{equation}\label{eq:M3-shell}
 M_3=\{\pm c\}\sqcup\{x+c/3:x\in X\}
                 \sqcup\{-x-c/3:x\in X\}.
\end{equation}
Thus $|M_3|=1640$.  The norm-$4$ shell is $L_4\sqcup(X-2c/3)\sqcup(-X+2c/3)$, so the counts in Proposition~\ref{prop:BBV-package} also give
\[
 |M_4|=|L_4|+2|X|=117936+2\cdot819=119574.
\]
Both counts agree with Borcherds' theta series \cite[Sec.~5.7]{BorcherdsThesis}; see also \cite{BacherVenkov}.  The passage between $L$ and $(M,c)$ is his rank-$27$ correspondence.

For the full series use a separate variable $t$, with exponent equal to the squared norm:
\[
 \Theta_M(t)=\sum_{z\in M}t^{z^2},\qquad
 \vartheta_3(t)=\sum_{n\in\Z}t^{n^2},\qquad
 \Delta_8(t)=t\prod_{n\ge1}(1-t^{2n-1})^8(1-t^{4n})^8.
\]
The standard theta-series expression for an odd unimodular rank-$27$ lattice is a linear combination of $\vartheta_3^{27-8j}\Delta_8^j$ for $0\le j\le3$ \cite{ConwaySloane}.  The constant term, absence of norms $1,2$, and the value $|M_3|=1640$ determine the four coefficients, giving
\begin{equation}\label{eq:theta-M-closed}
 \Theta_M(t)=\vartheta_3^{27}
       -54\vartheta_3^{19}\Delta_8
       +216\vartheta_3^{11}\Delta_8^2
       -1024\vartheta_3^3\Delta_8^3.
\end{equation}
In particular,
\begin{equation}\label{eq:theta-M-expansion}
 \begin{aligned}
 \Theta_M(t)={}&1+1640t^3+119574t^4+1497600t^5
                      +16733184t^6\\
             &{}+108081792t^7+588805308t^8+\cdots.
 \end{aligned}
\end{equation}
The gluing also gives
\begin{equation}\label{eq:theta-M-gluing}
 \Theta_M(t)=\theta_L(t^2)\sum_{n\in\Z}t^{3n^2}
       +2\theta_{C_+}(t^2)\sum_{n\in\Z}t^{3(n+1/3)^2},
\end{equation}
so the theta-character recursion \eqref{eq:theta-character-scalar}--\eqref{eq:theta-character-quadratic} recovers the same coefficients.  The value $1497600$ likewise agrees with Borcherds' norm-$5$ count \cite[Sec.~5.7]{BorcherdsThesis}.

Fix a line $\ell=\{x,y,z\}$ and put
\[
 u=x+c/3,\qquad v=y+c/3,\qquad w=z+c/3.
\]
Then
\begin{equation}\label{eq:line-in-M}
 u+v+w=c,\qquad u^2=v^2=w^2=3,\qquad
 (u,v)=(v,w)=(w,u)=-1.
\end{equation}
Define
\[
 S_\ell=\langle u,v,w\rangle_\Z,\qquad
 K_\ell=S_\ell^\perp\cap M.
\]
The Gram matrix of $S_\ell$ is $4I_3-J_3$, of determinant $16$.  Its rational span contains the orthonormal basis
\begin{equation}\label{eq:eu-ev-ew}
 e_u=\frac{u-c}{2}=-\frac{v+w}{2},\qquad
 e_v=\frac{v-c}{2},\qquad e_w=\frac{w-c}{2}.
\end{equation}
Write $I_\ell=\Z e_u\oplus\Z e_v\oplus\Z e_w\cong\Z^3$.

\begin{lemma}\label{lem:line-sublattice}
The lattice $S_\ell$ is primitive in $M$.  Its complement $K_\ell$ is even of rank $24$, determinant $16$, and minimum at least $4$.
\end{lemma}

\begin{proof}
In the orthonormal coordinates \eqref{eq:eu-ev-ew},
\[
 S_\ell=\{(a,b,d)\in\Z^3:a\equiv b\equiv d\pmod2\},
\]
so $[I_\ell:S_\ell]=4$ and
\[
 S_\ell^\#=\{(a,b,d)/2:a,b,d\in\Z,\ a+b+d\equiv0\pmod2\}.
\]
Every vector of $S_\ell^\#\setminus I_\ell$ has exactly two half-integral coordinates and hence norm in $\tfrac12+\Z$.  Therefore every integral overlattice of $S_\ell$ lies in $I_\ell$.  The three non-zero classes in $I_\ell/S_\ell$ are represented by $e_u,e_v,e_w$, each of norm $1$.  Any proper integral enlargement of $S_\ell$ thus contains a norm-$1$ vector.  Since $M$ has minimum $3$, $S_\ell$ is primitive in $M$.

Unimodularity of $M$ now gives $\det K_\ell=\det S_\ell=16$.  Also $c\in S_\ell$ is characteristic, so $a^2\equiv(a,c)=0\pmod2$ for $a\in K_\ell$.  Thus $K_\ell$ is even, and its inclusion in $M$ gives minimum at least $4$.
\end{proof}

\subsection{The Niemeier completion and its trio}

Let
\[
 \psi:S_\ell^\#/S_\ell\longrightarrow K_\ell^\#/K_\ell
\]
be the anti-isometry of discriminant bilinear forms defined by the gluing in $M$.  The subgroup $I_\ell/S_\ell\cong\F_2^2$ gives an integral unimodular completion
\begin{equation}\label{eq:Gamma-line}
 \Gamma_\ell=\{a\in K_\ell^\#:\bar a\in\psi(I_\ell/S_\ell)\},
 \qquad [\Gamma_\ell:K_\ell]=4.
\end{equation}

\begin{theorem}[line deletion]\label{thm:line-deletion}
The lattice $\Gamma_\ell$ is the Niemeier lattice $N(A_1^{24})$.  Its $48$ roots correspond to the points off $\ell$ collinear with a point of $\ell$.  The three points of $\ell$ partition its $24$ root pairs into a trio of Golay octads.
\end{theorem}

\begin{proof}
For $a\in X$, put $t=a+c/3\in M_3$.  The nearest-point property in Proposition~\ref{prop:scheme-hexagon-equivalence} shows that
\[
 ((t,u),(t,v),(t,w))
\]
is, up to permutation, one of
\[
 (3,-1,-1),\qquad(-1,1,1),\qquad(1,0,0).
\]
In the second case, if the triple is $(-1,1,1)$, the projection of $t$ to $S_\ell\otimes\R$ is $-e_u$.  Hence
\[
 r=t+e_u\in\Gamma_\ell,\qquad r^2=2.
\]
The same construction applies at $v$ and $w$.  There are eight further lines through each point of $\ell$, with two further points on each line.  Within each fibre the map $t\mapsto t+e_u$ (or its counterpart at $v$ or $w$) is injective.  The three fibres give the distinct classes $\psi(\bar e_u)$, $\psi(\bar e_v)$, and $\psi(\bar e_w)$ in $\Gamma_\ell/K_\ell$, so the construction gives $3\cdot8\cdot2=48$ distinct roots.  The two points on one such line give opposite roots.  In particular, all three non-zero classes of $\Gamma_\ell/K_\ell$ have norm-$2$ representatives.  Since $K_\ell$ is even and $\Gamma_\ell$ is integral, this also proves that $\Gamma_\ell$ is even.

Conversely, a root $r\in\Gamma_\ell$ cannot lie in $K_\ell$.  Its class corresponds to one of $e_u,e_v,e_w$ modulo $S_\ell$, say $e_u$.  Since that class has order $2$, the vector $t=-e_u+r$ belongs to $M$.  It has norm $3$ and $(t,c)=1$.  By \eqref{eq:M3-shell}, $t=a+c/3$ for some $a\in X$, and its projection $-e_u$ shows that $a$ is collinear with $x$ and lies off $\ell$.  Thus the displayed $48$ roots are all the roots.

One can identify the Niemeier lattice directly, without the full Niemeier classification.  For a harmonic quadratic polynomial $P$, the weighted theta series of $\Gamma_\ell$ is a cusp form of weight $14$, hence zero.  Its root shell is therefore a spherical $2$-design.  For any root $r$ this gives
\[
 \sum_{s\in(\Gamma_\ell)_2}(r,s)^2
      =\frac{48\cdot2}{24}\,r^2=8.
\]
The terms $s=\pm r$ already contribute $8$, so every other root is orthogonal to $r$.  The $48$ roots consequently form $24$ mutually orthogonal pairs, with root lattice $A_1^{24}$.  Even unimodularity makes its binary glue code doubly even and self-dual, and the absence of additional roots excludes weight $4$.  The Golay uniqueness theorem thus gives $\Gamma_\ell\cong N(A_1^{24})$ \cite[Chs.~11 and~16]{ConwaySloane}.

It remains to identify the partition as a Golay trio.  The quotient map
\[
 \lambda:\Gamma_\ell\longrightarrow\Gamma_\ell/K_\ell\cong\F_2^2
\]
labels the $24$ root components by the three non-zero elements, each eight times.  Choose roots $r_1,\ldots,r_{24}$, one per component, and let $\mathcal G$ be the binary Golay glue code of $\Gamma_\ell$.  For $g=(g_i)\in\mathcal G$, the vector $h_g=\tfrac12\sum_i g_i r_i$ lies in $\Gamma_\ell$.  Therefore
\[
 \sum_i g_i\lambda(r_i)=\lambda(2h_g)=0.
\]
For every non-zero linear functional $\chi$ on $\F_2^2$, the binary word $(\chi(\lambda(r_i)))_i$ is orthogonal to $\mathcal G$ and hence lies in $\mathcal G$, since this code is self-dual.  It has weight $16$; its complement is the fibre of the unique non-zero element in $\ker\chi$.  The all-one word is in $\mathcal G$, so this complement is a Golay octad.  The three fibres are therefore a trio.
\end{proof}

The same Niemeier lattice also has the familiar Leech-neighbour description.  Choose one root $r_i$ in each component and put
\[
 \rho=\frac12\sum_{i=1}^{24}r_i\in\Gamma_\ell,
 \qquad \rho^2=12.
\]
Then
\begin{equation}\label{eq:Leech-neighbour}
 \Lambda_\ell
 =\{a\in\Gamma_\ell:(a,\rho)\equiv0\pmod2\}
 \ \cup\ 
 \left(\frac\rho2+\{a\in\Gamma_\ell:(a,\rho)\equiv1\pmod2\}\right)
\end{equation}
is even and unimodular.  The first coset contains no root, since $(r_i,\rho)=1$.  Vectors in the second coset have odd-quarter root coordinates, so their norm is at least $24\cdot2/16=3$; evenness again excludes roots.  Thus $\Lambda_\ell$ is the Leech lattice \cite[Ch.~24]{ConwaySloane}.

This gives the link with the $A_1^{24}$ deep-hole picture, and a flag on $\ell$ singles out one octad of the trio.  The neighbour involves a choice of root signs, however.  We do not impose a fixed Leech-neighbour polarization on the reconstruction datum: the intrinsic objects are $K_\ell\subset\Gamma_\ell$, the trio, and the full discriminant gluing.  The order calculation below consequently uses all sign changes of $N(A_1^{24})$, not only the Golay sign changes preserving a chosen neighbour.

\subsection{The quotient map and its admissibility}
\label{sec:deletion-data}

We now describe exactly what must be retained after deletion.  Fix a Niemeier lattice $\Gamma=N(A_1^{24})$, with roots $\{\pm r_i\}_{i=1}^{24}$ and root lattice
\[
 R=\bigoplus_{i=1}^{24}\Z r_i,
 \qquad r_i^2=2.
\]
Its Golay description is
\begin{equation}\label{eq:niemeier-golay-model}
 \Gamma=\left\{\frac12\sum_{i=1}^{24}a_i r_i:
       a_i\in\Z,\ (a_i\bmod2)\in\mathcal G_{24}\right\}.
\end{equation}
Fix a trio $\mathcal T=\{O_1,O_2,O_3\}$ and write $H_2=\F_2^2$, with non-zero elements $h_1,h_2,h_3$ satisfying $h_1+h_2+h_3=0$.  The subscripts are bookkeeping labels; the trio itself is not ordered.

\begin{definition}[the admissible quotient datum]\label{def:quotient-datum}
An admissible rank-$24$ quotient datum consists of
\[
 (\Gamma,\mathcal T,\lambda),
\]
where $\lambda:\Gamma\to H_2$ is a homomorphism satisfying
\begin{equation}\label{eq:lambda-roots}
 \lambda(r_j)=h_i\qquad(j\in O_i),
\end{equation}
and the following admissibility condition.  For $\chi\in H_2^*$, let $\bar t_\chi\in\Gamma/2\Gamma$ be the unique class with
\[
 (t_\chi,a)\equiv\chi(\lambda(a))\pmod2
       \qquad(a\in\Gamma).
\]
Unimodularity of $\Gamma$ makes this class well defined.  Put
\[
 q_2(\bar t)=\frac{t^2}{2}\pmod2.
\]
We require
\begin{equation}\label{eq:lambda-admissible}
 q_2(\bar t_\chi)=1\qquad(0\ne\chi\in H_2^*).
\end{equation}
\end{definition}

The root condition makes $\lambda$ surjective.  Its kernel $K$ is even, has determinant $16$, and contains no roots, because none of the roots of $\Gamma$ maps to zero.  The values on the roots do not determine $\lambda$ on the Golay glue vectors.  It is this extension, together with the compatible discriminant gluing described below, that plays the role of the retained deep-hole coset in the code construction.

The number of possible extensions can be obtained without listing them.  This is the finite count needed for the group order.

\begin{lemma}[the admissible quotient maps]\label{lem:admissible-map-count}
For a fixed trio and the root labels \eqref{eq:lambda-roots}, there are exactly $2^{21}$ admissible maps $\lambda:\Gamma\to H_2$.
\end{lemma}

\begin{proof}
Let $\chi_i$ be the non-zero character vanishing on $h_i$, and take $\chi_1,\chi_2$ as a basis of $H_2^*$.  Since $2\Gamma\subset R$ and $[\Gamma:R]=2^{12}$, the group
\[
 B=R/2\Gamma
\]
is a $12$-dimensional binary vector space.  For $b=\sum_j n_jr_j\in R$, define
\[
 f_i(\bar b)=\sum_{j\in O_i}n_j\pmod2.
\]
These functions are well defined on $B$: every Golay word has even intersection with each octad.  Moreover, $f_1,f_2,f_3$ are independent, as is seen by evaluating them on one root from each octad.

A map with the prescribed root labels is equivalent to a pair of classes $\bar t_1,\bar t_2\in\Gamma/2\Gamma$ representing $\chi_1\lambda$ and $\chi_2\lambda$.  All such pairs have the form
\[
 t_1=t_1^0+b_1,\qquad t_2=t_2^0+b_2,
 \qquad \bar b_1,\bar b_2\in B,
\]
where
\[
 t_1^0=\frac12\sum_{j\in O_2\cup O_3}r_j,
 \qquad
 t_2^0=\frac12\sum_{j\in O_1\cup O_3}r_j.
\]
Indeed, pairing with the roots fixes precisely the parity of the coefficients in \eqref{eq:niemeier-golay-model}, and two representatives with the same parities differ by an element of $R$.  Thus there are $2^{24}$ extensions before admissibility is imposed.

Both $t_i^0$ have norm $8$, and $(t_1^0,t_2^0)=4$.  Also
\[
 q_2(\bar b)=f_1(\bar b)+f_2(\bar b)+f_3(\bar b)
       \qquad(b\in R).
\]
Using $q_2(s+t)=q_2(s)+q_2(t)+(s,t)$ modulo $2$, we obtain
\[
 q_2(\bar t_1)=f_1(\bar b_1),\qquad
 q_2(\bar t_2)=f_2(\bar b_2).
\]
The third non-zero character is $\chi_1+\chi_2$, represented by $\bar t_1+\bar t_2$.  Once the first two quadratic values are $1$, its quadratic value is $1$ exactly when $(t_1,t_2)$ is odd.  The latter condition reduces to
\[
 f_3(\bar b_1)+f_3(\bar b_2)=1.
\]
Admissibility is therefore the system of three independent affine equations
\begin{equation}\label{eq:three-admissibility-equations}
 f_1(\bar b_1)=1,\qquad
 f_2(\bar b_2)=1,\qquad
 f_3(\bar b_1)+f_3(\bar b_2)=1.
\end{equation}
It has $2^{24-3}=2^{21}$ solutions.
\end{proof}

\subsection{Oriented gluing and reconstruction}
\label{sec:rank24-reconstruction}

We next restore the deleted line.  In the deletion notation \eqref{eq:eu-ev-ew}, use the orthonormal basis $e_i=-e_{u_i}$, so that
\[
 c=e_1+e_2+e_3,\qquad u_i=c-2e_i.
\]
This choice keeps $c=u_1+u_2+u_3$.  Abstractly, put $I_3=\bigoplus_i\Z e_i$ and
\[
 S=\langle u_1,u_2,u_3\rangle
   =\{(n_1,n_2,n_3)\in\Z^3:n_1\equiv n_2\equiv n_3\pmod2\}.
\]
Its three non-zero classes in $I_3/S$ are represented by the $e_i$ and are labelled by the three octads.  The characteristic refinement on $A_S=S^\#/S$ is
\[
 q_{S,c}(\bar x)=x^2-(x,c)\pmod{2\Z}.
\]
For $K=\ker\lambda$, the even discriminant form is $q_K(\bar y)=y^2\pmod{2\Z}$.

\begin{proposition}[rank-$24$ extension and recovery]\label{prop:rank24-extension-recovery}
For every admissible map $\lambda$, there are exactly two anti-isometries
\[
 \varphi:(A_S,q_{S,c})\longrightarrow(A_K,q_K)
\]
that identify the class of $e_i$ with the class of a root in $O_i$.  Either choice $\varphi$ gives an odd unimodular rank-$27$ lattice
\begin{equation}\label{eq:Mphi-construction}
 M_\varphi
 =\{(x,y)\in S^\#\oplus K^\#:
                  \bar y=\varphi(\bar x)\}
\end{equation}
of minimum $3$, in which $c$ is characteristic of norm $3$.  Its complement
\[
 L=c^\perp\cap M_\varphi
\]
is even of rank $26$, determinant $3$, and minimum $4$.  The vectors $u_i-c/3$ form a line of its recovered hexagon.  For a quotient map and gluing obtained by line deletion, this construction recovers the original marked pair $(M,c)$ and line.
\end{proposition}

\begin{proof}
The dual of $S$ is
\[
 S^\#=\left\{\frac12(m_1,m_2,m_3):
                   m_i\in\Z,\ \sum_i m_i\equiv0\pmod2\right\}.
\]
Thus $A_S\cong(\Z/4)^2$ and $H_S=I_3/S\cong\F_2^2$.  With cyclic subscripts, set
\[
 \delta_i=\frac{u_j-u_k}{4}=\frac{e_k-e_j}{2}.
\]
In $A_S$ these classes satisfy
\[
 2\delta_i=\bar e_i,\qquad
 \delta_1+\delta_2+\delta_3=0,\qquad
 q_{S,c}(\delta_i)=\frac12.
\]
The pairing of $\delta_i$ with $\bar e_i$ is zero, and its pairing with either other $\bar e_j$ is $1/2$ modulo $\Z$.

On the $K$ side, write $H_K=\Gamma/K=\{0,h_1,h_2,h_3\}$, identifying these classes with $H_2$ through $\lambda$.  There is an exact sequence
\begin{equation}\label{eq:AK-exact}
 0\longrightarrow H_K\longrightarrow A_K
 \xrightarrow{\ \pi\ }H_2^*\longrightarrow0,
 \qquad \pi(\bar y)(h)=2b_K(\bar y,h).
\end{equation}
The fibre over $\chi_i$ is represented by $t_{\chi_i}/2+\Gamma$.  Admissibility makes the three non-zero vectors $\bar t_\chi$ an anisotropic plane for $q_2$: their distinct pairings are $1$ modulo $2$.  Consequently
\[
 \lambda(t_{\chi_i})=h_i,
\]
since $\chi_j(\lambda(t_{\chi_i}))=(t_{\chi_j},t_{\chi_i})\pmod2$.  Every class in $\pi^{-1}(\chi_i)$ therefore doubles to $h_i$.

The four classes in this fibre have half-integral norm.  Adding $h_i$ leaves the norm modulo $2\Z$ unchanged, while adding either other $h_j$ changes it by $1$.  Exactly two classes consequently have norm $3/2$ modulo $2\Z$.  They are negatives of one another; denote them by $\pm\eta_i$.  Their signs can be chosen so that
\begin{equation}\label{eq:eta-orientation}
 \eta_1+\eta_2+\eta_3=0.
\end{equation}
To see both existence and the number of choices, first choose $\eta_1$.  The pairing $b_K(\eta_1,\eta_2)$ is $1/4$ or $3/4$ modulo $\Z$, because $2\eta_1=h_1$ and $b_K(h_1,\eta_2)=1/2$.  Exactly one sign of $\eta_2$ makes $q_K(\eta_1+\eta_2)=3/2$; then $\eta_3=-\eta_1-\eta_2$ is forced.  There are precisely two resulting triples, related by simultaneous negation.

Define $\varphi(\delta_i)=\eta_i$, hence $\varphi(\bar e_i)=h_i$.  The relations show that $\varphi$ is an isomorphism, and the full refinement check is
{\renewcommand{\arraystretch}{1.3}
\[
\begin{array}{c|c|c|c}
 a&q_{S,c}(a)&\varphi(a)&q_K(\varphi(a))\\ \hline
 0&0&0&0\\
 \bar e_i&0&h_i&0\\
 \delta_i&\frac12&\eta_i&\frac32\\
 \delta_i+\bar e_i&\frac12&-\eta_i&\frac32\\
 \delta_i+\bar e_j\ (j\ne i)&\frac32&\eta_i+h_j&\frac12
\end{array}
\]
}
with $1+3+3+3+6=16$ classes.  Thus
\begin{equation}\label{eq:refinement-antiisometry}
 q_K(\varphi(a))=-q_{S,c}(a),
\end{equation}
and polarization gives the corresponding anti-isometry of discriminant bilinear forms.  Conversely, every compatible anti-isometry must send the $\delta_i$ to one of the two triples just described.

The graph gluing \eqref{eq:Mphi-construction} has index $16$ over $S\oplus K$, whose determinant is $16^2$.  It is therefore integral and unimodular.  The refinement identity gives
\[
 (x,y)^2-((x,y),c)
   \equiv q_{S,c}(\bar x)+q_K(\varphi(\bar x))
   \equiv0\pmod{2\Z},
\]
so $c$ is characteristic and the lattice is odd.

It is important that every admissible map, not only one coming from a previously known lattice, gives a rootless extension.  If $x\in S$, then $y\in K$, and the non-zero vectors in these two lattices have norms at least $3$ and $4$.  If $x\in I_3\setminus S$, then $x^2\ge1$ and $y\in\Gamma\setminus K$ has $y^2\ge2$.  In the remaining case, $x\in S^\#\setminus I_3$ has $x^2\ge1/2$, and $y\in K^\#\setminus\Gamma$ lies in a coset $t_{\chi_i}/2+\Gamma$.  In root coordinates its numerator, with denominator $4$, is odd on the $16$ coordinates of $O_j\cup O_k$, because these are exactly the roots on which $\chi_i\lambda$ is non-zero.  Hence
\[
 y^2\ge16\cdot\frac{2}{16}=2.
\]
The total norm is at least $5/2$, and integrality raises it to at least $3$.  Since each $u_i$ has norm $3$, the minimum of $M_\varphi$ is exactly $3$.

The vector $c$ is primitive because $c^2=3$, so its complement has determinant $3$; it is even because $c$ is characteristic.  Its minimum is at least $4$ and is attained by $r_a-r_b\in K$ for distinct components $a,b$ in one octad.  The three vectors $u_i-c/3$ have norm $8/3$, lie in the same non-zero discriminant class of $L$, and sum to zero.  They therefore form a hexagon line.  Finally, for a quotient map and gluing obtained by deletion from $M$, the primitive-complement anti-isometry is the chosen $\varphi$, and the graph gluing gives back $M$ with its marked $c$ and line.
\end{proof}

\begin{definition}[the oriented deletion datum]\label{def:deletion-datum}
An oriented rank-$24$ deletion datum is
\[
 \mathcal D=(\Gamma,\mathcal T,\lambda,\epsilon),
\]
where $(\Gamma,\mathcal T,\lambda)$ is an admissible quotient datum and $\epsilon=\varphi$ is one of the two compatible anti-isometries in Proposition~\ref{prop:rank24-extension-recovery}.  Isomorphisms of data may permute the three octads and the corresponding non-zero elements of $H_2$; they must intertwine the quotient maps and the full gluings, using the induced isometry of $S$.
\end{definition}

The reconstruction also keeps track of isomorphisms, just as extension from the odd Golay code and its deep-hole coset does in the code paper.

\begin{corollary}[the line-marked correspondence]\label{cor:line-marked-correspondence}
Line deletion and reconstruction give inverse correspondences, including isomorphisms, between oriented rank-$24$ data and the marked triples $(M,c,\ell)$ above.  In particular, with $G=\Aut^+(L)$,
\begin{equation}\label{eq:line-data-stabilizer}
 G_\ell\cong\Aut(\Gamma_\ell,\mathcal T_\ell,
                         \lambda_\ell,\epsilon_\ell).
\end{equation}
\end{corollary}

\begin{proof}
For a deleted line, the gluing identifies $\delta_i$ with a class $\eta_i$ of norm $3/2$ modulo $2\Z$.  Pairing with the three root classes gives $\pi(\eta_i)=\chi_i$.  Thus a representative $y_i\in K_\ell^\#$ of $\eta_i$ lies in $t_{\chi_i}/2+\Gamma_\ell$, so
\[
 2y_i\equiv t_{\chi_i}\pmod{2\Gamma_\ell},
 \qquad q_2(\overline{2y_i})=2y_i^2\equiv1\pmod2.
\]
Thus the quotient map satisfies admissibility, and Proposition~\ref{prop:rank24-extension-recovery} applies.  Conversely, its construction recovers $K$, the completion $\Gamma$, the three root fibres, and the full gluing.  This proves the inverse correspondence on objects.

An isometry fixing $c$ and preserving $\ell$ permutes its three $u_i$, hence the corresponding $e_i$, and preserves $S$, $I_3$, $K$, and $\Gamma$.  Its restriction respects the quotient and gluing.  This restriction is faithful: if it is the identity on $\Gamma$, then it fixes the three quotient fibres and hence each $u_i$, so it is the identity on both orthogonal summands $S\otimes\R$ and $K\otimes\R$.

Conversely, an isomorphism of data induces a permutation of the three octads and therefore an isometry of $S$ fixing $c$.  Together with the isometry of $K$, it preserves the graph of $\varphi$ in \eqref{eq:Mphi-construction}.  It consequently extends uniquely to an isometry of $(M,c,\ell)$.  Restricting to $c^\perp$ proves \eqref{eq:line-data-stabilizer}.
\end{proof}

\subsection{The eight sign orbits and positive-definite uniqueness}
\label{sec:golay-orbit}

Counting the admissible extensions does not by itself prove uniqueness: the extensions could lie in different orbits and reconstruct different lattices.  We now resolve this issue inside the Golay description.  Root sign changes leave eight classes of quotient maps.  An explicit action on these eight classes proves that they all give the same line-marked lattice.

For the calculation, regard the $24$ coordinates as three copies of $\F_2^3$, one for each octad.  Write $x,y,z$ for the coordinate functions on $\F_2^3$, and identify a binary word with a triple of functions evaluated on these three copies.  Put
\[
 o_1=(1,0,0),\qquad o_2=(0,1,0),\qquad o_3=(0,0,1).
\]
We use the following generators for the Golay code $\mathcal C$:
\begin{equation}\label{eq:golay-trio-generators}
 \begin{gathered}
 o_1,\ o_2,\ o_3,\qquad
 (l,l,0),\ (l,0,l)\quad(l=x,y,z),\\
 a=(xy,xy,xy+x+z),\\
 b=(yz,yz,yz+x+y),\\
 d=(zx,zx,zx+y+z).
 \end{gathered}
\end{equation}
Equivalently, its words are the triples
\begin{equation}\label{eq:golay-polynomial-form}
 (Q+l_1+c_1,\ Q+l_2+c_2,\ Q+l_3+c_3),
 \qquad l_1+l_2+l_3=\kappa(Q),
\end{equation}
where $Q\in\langle xy,yz,zx\rangle$, the $l_i$ are linear forms, the $c_i$ are constants, and
\[
 \kappa(xy)=x+z,\qquad
 \kappa(yz)=x+y,\qquad
 \kappa(zx)=y+z.
\]
In particular, every permutation of the three blocks preserves the code.

For completeness, the twelve generators are independent, mutually orthogonal, and of weight $8$.  They therefore span a doubly even self-dual code.  Every word has the same homogeneous quadratic part in its three blocks.  If this part is non-zero, each block has weight at least $2$, so divisibility by four makes the total weight at least $8$.  If the quadratic part is zero, the three linear parts sum to zero.  A non-constant linear part then occurs in at least two blocks, again giving weight at least $8$; words with only constants have weight divisible by $8$.  Thus this is a $[24,12,8]$ Golay code, with the three blocks as a trio.  The uniqueness of the Golay code and its transitivity on trios allow these coordinates for our fixed $(\Gamma,\mathcal T)$ \cite[Ch.~11]{ConwaySloane}.

\begin{proposition}[transitivity on the complete deletion data]\label{prop:deletion-data-transitive}
Let $\mathscr D$ be the oriented admissible deletion data for a fixed $(\Gamma,\mathcal T)$ and root labels.  The group of independent root sign changes has eight orbits on $\mathscr D$, each of size $2^{19}$.  The full group $\Aut(\Gamma)_{\mathcal T}$ is transitive on $\mathscr D$.
\end{proposition}

\begin{proof}
Identify $B=R/2\Gamma$ with $\F_2^{24}/\mathcal C$, using root coefficients, and put
\[
 g_1=o_2+o_3,\qquad g_2=o_1+o_3.
\]
A quotient map is represented as before by a pair $(\bar b_1,\bar b_2)\in B^2$.  A root sign change with support $s\in\F_2^{24}$ sends this pair to
\begin{equation}\label{eq:root-sign-action}
 (\bar b_1+\overline{s g_1},\ \bar b_2+\overline{s g_2}),
\end{equation}
where juxtaposition denotes coordinatewise multiplication of binary words.  This follows by applying the sign change to $t_i^0+b_i$: its effect on $b_i$ is trivial modulo $2\Gamma$, and its effect on $t_i^0$ is the indicated root vector modulo $2\Gamma$.

To describe the orbits, consider the binary space
\[
 \mathcal W=\{w\in\F_2^{24}:wg_1\in\mathcal C,\ wg_2\in\mathcal C\}.
\]
The generators \eqref{eq:golay-trio-generators} show that
\begin{equation}\label{eq:common-affine-space}
 \mathcal W=\{(l+c_1,l+c_2,l+c_3):
          l\in\langle x,y,z\rangle,\ c_i\in\F_2\}.
\end{equation}
Indeed, a codeword with one block zero has zero quadratic part, and its other two blocks have the same linear part.  Applying this first to $wg_1=(0,w_2,w_3)$ and then to $wg_2=(w_1,0,w_3)$ gives \eqref{eq:common-affine-space}; the converse follows from the affine generators.  In particular, $\dim\mathcal W=6$.

For $p=(\bar b_1,\bar b_2)$ define a functional on $\mathcal W$ by
\begin{equation}\label{eq:sign-orbit-functional}
 F_p(w)=b_1\mathbin{\cdot}(wg_2)+b_2\mathbin{\cdot}(wg_1),
\end{equation}
where the dot is the binary coordinate pairing.  The functional is well defined because $wg_i\in\mathcal C$ and $\mathcal C$ is self-dual.  It is a complete invariant for the sign action, not merely an invariant.  To see this, the dual of $B^2$ is $\mathcal C^2$.  The annihilator of the image in \eqref{eq:root-sign-action} is
\[
 \{(c_1,c_2)\in\mathcal C^2:c_1g_1+c_2g_2=0\}
       =\{(wg_2,wg_1):w\in\mathcal W\}.
\]
Thus $p\mapsto F_p$ is onto $\mathcal W^*$ and its kernel is exactly the image of the sign action.  That image has dimension $24-6=18$.

The three admissibility equations \eqref{eq:three-admissibility-equations} become
\[
 F_p(o_1)=F_p(o_2)=F_p(o_3)=1.
\]
By \eqref{eq:common-affine-space}, the functionals satisfying these equations are indexed by $v\in\F_2^3$, with
\begin{equation}\label{eq:eight-sign-orbits}
 F_p(l+c_1,l+c_2,l+c_3)
       =l(v)+c_1+c_2+c_3.
\end{equation}
There are therefore exactly eight sign orbits on the admissible quotient maps, each of size $2^{18}$.

It remains to check that coordinate permutations join these eight orbits.  The following two permutations preserve each octad and the code:
\begin{equation}\label{eq:two-golay-permutations}
 \begin{aligned}
 \sigma:&\quad (x,y,z)\longmapsto(z,x,y)
                      &&\text{on all three blocks},\\
 \tau:&\quad (x,y,z)\longmapsto(x+y,y,z)
                      &&\text{on the first two blocks},\\
      &\quad (x,y,z)\longmapsto(x+y,y,z+1)
                      &&\text{on the third block}.
 \end{aligned}
\end{equation}
Here the displayed maps act on coordinate positions, and a coordinate permutation $\pi$ acts on a word $f$ by $(\pi\cdot f)(\alpha)=f(\pi^{-1}\alpha)$.  The permutation $\sigma$ cyclically permutes $a,b,d$ and preserves the affine generators.  The permutation $\tau$ preserves the affine-generator space and satisfies
\[
 \tau(a)=a+(y,y,0)+o_3,\qquad
 \tau(b)=b,\qquad
 \tau(d)=b+d+o_3.
\]
These identities directly verify preservation of $\mathcal C$.

These affine coordinates describe the whole octad-wise stabilizer, not just a useful subgroup.  Let $T_0$ be the subgroup of $M_{24}$ fixing each of the three octads.  It preserves $\mathcal W$.  On each octad the restrictions of $\mathcal W$ are all affine functions, so every element of $T_0$ is affine on that octad; the equality of the three linear parts in \eqref{eq:common-affine-space} forces a common linear part.  Thus every element has the form
\[
 \alpha\longmapsto A\alpha+t_i\quad\hbox{on }O_i,
 \qquad A\in\operatorname{GL}_3(2),\quad t_i\in\F_2^3.
\]
The kernel of the linear-part map is exactly
\begin{equation}\label{eq:trio-translation-kernel}
 U=\{(t_1,t_2,t_3)\in(\F_2^3)^3:t_1+t_2+t_3=0\}
       \cong(C_2)^6.
\end{equation}
Indeed, a translation changes the sum of the linear parts in \eqref{eq:golay-polynomial-form} by $B_Q(t_1+t_2+t_3,\,\cdot)$, where $B_Q$ is the polar form of $Q$.  The polar forms of $xy,yz,zx$ have zero common radical, so preservation of the code is equivalent to the displayed sum-zero condition.  The linear parts of $\sigma,\tau$ generate $\operatorname{GL}_3(2)$: conjugating the elementary transvection by the coordinate cycle gives the cyclic elementary transvections, and their commutators give the remaining ones.  Consequently
\[
 T_0/U\cong L_3(2),\qquad
 |T_0|=2^6\cdot168=10752.
\]

Substitution in \eqref{eq:sign-orbit-functional} gives the action of every element of $T_0$ on the eight parameters:
\begin{equation}\label{eq:full-affine-eight-action}
 v\longmapsto Av+t_1+t_2+t_3.
\end{equation}
Its kernel is exactly $U$.  The two displayed permutations consequently act by
\begin{equation}\label{eq:residual-eight-point-action}
 \sigma(v_1,v_2,v_3)=(v_3,v_1,v_2),\qquad
 \tau(v_1,v_2,v_3)=(v_1+v_2,v_2,v_3+1).
\end{equation}
Their orbit of zero contains all eight vectors: $\tau$ first gives $001$; powers of $\sigma$ give the other unit vectors; $\tau(010)=111$ and $\tau(100)=101$; powers of $\sigma$ give $011$ and $110$.  Thus the quotient maps form one orbit already under the octad-wise subgroup, together with root sign changes.

Finally, $-1\in\Aut(\Gamma)$ is a root sign change that fixes every quotient map and interchanges its two gluings.  On $A_K$ it sends each of the order-four classes $\eta_i$ to $-\eta_i$, so it sends $\varphi$ to the other compatible gluing $-\varphi$.  Hence each sign orbit on quotient maps lifts to one orbit of twice the size on the oriented data.  There are eight such orbits of size $2^{19}$, and the coordinate permutations join them into one orbit, as claimed.
\end{proof}

We can now prove uniqueness rather than using it as an input to the rank-$24$ argument.  The point is that the orbit calculation applies to all admissible reconstruction data, not just to data extracted from a previously fixed lattice.

\begin{theorem}[positive-definite existence and uniqueness]\label{thm:positive-definite-uniqueness}
Up to isometry there is exactly one even positive-definite lattice of rank $26$, determinant $3$, and minimum $4$.  The associated triples $(M,c,\ell)$ are unique up to isometry, and $G=\Aut^+(L)$ is transitive on its lines, flags, and points.  The stabilizer of any line induces the full $S_3$ on its three points.
\end{theorem}

\begin{proof}
Admissible quotient maps exist by Lemma~\ref{lem:admissible-map-count}, and Proposition~\ref{prop:rank24-extension-recovery} constructs a lattice with the stated invariants from any such map and a compatible gluing.  Conversely, every lattice with these invariants has the hexagon constructed in Section~\ref{sec:819-hexagon}.  Deleting any line gives admissible data on $N(A_1^{24})$ with a Golay trio.  After an isometry of the Niemeier lattice and its trio, all these data use the same $(\Gamma,\mathcal T)$.  Proposition~\ref{prop:deletion-data-transitive} makes them isomorphic, and Corollary~\ref{cor:line-marked-correspondence} lifts this isomorphism to the triples $(M,c,\ell)$.  Restriction to $c^\perp$ proves lattice uniqueness.  Applying the same argument to two lines in one fixed lattice gives an isometry fixing $c$ and carrying one line to the other; its restriction lies in $\Aut^+(L)$.

For the local action on a line, let $A=\Aut(\Gamma)_{\mathcal T}$ and let $A_0$ fix each octad of $\mathcal T$ individually.  All elements used to prove transitivity on the oriented data---root sign changes, $\sigma$, $\tau$, and $-1$---belong to $A_0$.  Thus $A_0$ is itself transitive on those data.  For any datum $d$,
\[
 A=A_0\operatorname{Stab}_A(d),\qquad
 \operatorname{Stab}_A(d)\twoheadrightarrow A/A_0\cong S_3.
\]
The quotient is the full $S_3$ because pure block permutations preserve the Golay model.  Under Corollary~\ref{cor:line-marked-correspondence}, the three octads are the three points of the deleted line, so $G_\ell$ induces this full symmetric group.  Combining it with line transitivity gives flag transitivity, and hence point transitivity.
\end{proof}

\begin{remark}[comparison with the definite unimodular classification]\label{rem:rank27-comparison}
The same uniqueness statement follows from the Bacher--Venkov classification of rootless unimodular rank-$27$ lattices with a characteristic vector of norm $3$, as recorded explicitly by King \cite[Example~7]{King}; see also \cite{BacherVenkov,ChenevierHunting}.  The argument above isolates the relevant case in the positive-definite Niemeier picture and does not require the complete rank-$27$ or rank-$28$ enumeration.  There is no ambiguity from the choice of a shortest characteristic vector: if $c,d$ are such vectors in a rootless unimodular lattice of minimum $3$ and $d\ne\pm c$, then $(c+d)/2$ and $(c-d)/2$ are non-zero lattice vectors, since characteristic vectors are congruent modulo $2M$.  Their squared norms sum to $3$, a contradiction.
\end{remark}

\subsection{Counting the order from the Niemeier data}
\label{sec:rank24-order}

Having proved uniqueness and flag transitivity, we now calculate the order directly from the same deletion data.  We give the framed-deletion count as well, since it exhibits the analogy with deletion and extension in the code construction.  Fix the model $(\Gamma,\mathcal T)$ used above and let
\[
 A=\Aut(\Gamma)_{\mathcal T}
\]
be its full trio stabilizer.  Since the roots span the ambient space, the Golay description gives
\[
 \Aut(\Gamma)=2^{24}:M_{24},\qquad
 |A|=2^{24}\,|(M_{24})_{\mathcal T}|.
\]
The Mathieu group is transitive on trios.  There are $759$ octads and $15$ trios through each octad, hence
\begin{equation}\label{eq:number-trios}
 \#\{\text{trios}\}=\frac{759\cdot15}{3}=3795,
 \qquad |(M_{24})_{\mathcal T}|=\frac{|M_{24}|}{3795}=64512;
\end{equation}
see \cite[Ch.~11, Sec.~12]{ConwaySloane}.

Let $\mathscr D$ be the set of oriented data with this fixed Niemeier lattice, trio, and root labels.  Lemma~\ref{lem:admissible-map-count} and Proposition~\ref{prop:rank24-extension-recovery} give
\begin{equation}\label{eq:number-oriented-data}
 |\mathscr D|=2^{21}\cdot2=2^{22}.
\end{equation}
The group $A$ acts on $\mathscr D$, allowing the induced relabelling of $H_2$ when its elements permute the octads.  Its transitivity was proved in Proposition~\ref{prop:deletion-data-transitive}.

\begin{theorem}[order calculation from the rank-$24$ picture]\label{thm:rank24-order-uniqueness}
For the rank-$26$ lattice $L$ and its recovered hexagon,
\begin{equation}\label{eq:G-order-rank24}
 \begin{aligned}
 |G|=|\Aut^+(L)|
   &=2457\,\frac{2^{24}\cdot64512}{2^{22}}
     =2457\cdot258048\\
   &=634023936=2^{12}3^5 7^2 13.
 \end{aligned}
\end{equation}
Consequently
\begin{equation}\label{eq:AutL-order-rank24}
 |\Aut(L)|=1268047872=2^{13}3^5 7^2 13.
\end{equation}
\end{theorem}

\begin{proof}
By Theorem~\ref{thm:positive-definite-uniqueness}, all pairs $(M,c)$ reconstructed from admissible data are isometric, even with a line marked.  We fix one such pair and put $L=c^\perp\cap M$.

A \emph{framed line deletion} consists of a line $\ell$ of the fixed hexagon and an isometry
\[
 f:(\Gamma_\ell,\mathcal T_\ell)
         \longrightarrow(\Gamma,\mathcal T).
\]
For each line there are exactly $|A|$ such isometries, because $\Gamma_\ell\cong N(A_1^{24})$ and $M_{24}$ is transitive on trios.  The set $\mathscr F$ of framed line deletions therefore has cardinality
\[
 |\mathscr F|=2457\,|A|.
\]
The group $G$ acts freely on $\mathscr F$.  Indeed, an element fixing a framing acts trivially on $\Gamma_\ell$, and the faithfulness assertion of Corollary~\ref{cor:line-marked-correspondence} makes it the identity.

Transporting the quotient map and gluing through $f$ associates an element of $\mathscr D$ with every framed line deletion.  Every datum occurs: reconstruct its pair $(M,c)$ and marked line, identify the pair with the fixed one by lattice uniqueness, and use the given model $\Gamma$ as the framing.  Moreover, two framed deletions give the same datum precisely when they are related by $G$, by the isomorphism statement of Corollary~\ref{cor:line-marked-correspondence}.  Hence
\[
 \mathscr F/G\ \cong\ \mathscr D,
 \qquad
 \frac{2457\,|A|}{|G|}=2^{22}.
\]
Substituting $|A|=2^{24}\cdot64512$ proves \eqref{eq:G-order-rank24}.  Proposition~\ref{prop:aut-identification-lattice} supplies the central factor $\{\pm1\}$ and gives \eqref{eq:AutL-order-rank24}.
\end{proof}

Flag transitivity was obtained in Theorem~\ref{thm:positive-definite-uniqueness}.  Consequently Theorem~\ref{thm:rank24-order-uniqueness} gives
\begin{equation}\label{eq:point-flag-orders-rank24}
 |G_p|=\frac{|G|}{819}=774144,\qquad
 |G_\ell|=\frac{|G|}{2457}=258048,\qquad
 |G_{p,\ell}|=\frac{|G|}{819\cdot9}=86016.
\end{equation}

\begin{proposition}[the projective line and the line stabilizer]\label{prop:eight-point-parabolic}
Let $\Omega$ be the eight root-sign orbits on the oriented deletion data, and put $T=(M_{24})_{\mathcal T}$.  With $U$ as in \eqref{eq:trio-translation-kernel},
\[
 T_0\cong 2^6:L_3(2),\qquad
 T\cong 2^6:\bigl(L_3(2)\times S_3\bigr).
\]
The induced $L_3(2)\cong L_2(7)$ action identifies $\Omega$ equivariantly, but noncanonically, with $\mathbf P^1(\F_7)$.  The block-permuting $S_3$ acts trivially on $\Omega$.  For $\omega\in\Omega$,
\begin{equation}\label{eq:trio-orbit-stabilizer}
 (T_0)_\omega\cong 2^6:(7:3),\qquad
 T_\omega\cong 2^6:\bigl((7:3)\times S_3\bigr),
\end{equation}
of orders $1344$ and $8064$, respectively.

For a line $\ell$ of the recovered hexagon, the lattice stabilizer fits into an exact sequence
\begin{equation}\label{eq:line-stabilizer-exact}
 1\longrightarrow(C_2)^5\longrightarrow G_\ell
   \xrightarrow{\ \rho\ }T_\omega\longrightarrow1.
\end{equation}
In particular, $P_\ell=\rho^{-1}(U)$ is a normal $2$-subgroup with
\begin{equation}\label{eq:line-parabolic-coarse}
 |P_\ell|=2^{11},\qquad
 G_\ell/P_\ell\cong(7:3)\times S_3.
\end{equation}
The line stabilizer is solvable.  The $S_3$ quotient in \eqref{eq:line-parabolic-coarse} is its action on the three points of $\ell$.
\end{proposition}

\begin{proof}
For an element of $T_0$ with linear part $A$, the sum $s(A)=t_1+t_2+t_3$ depends only on $A$, since two lifts differ by $U$.  Multiplying by a suitable element of $U$ gives a lift with $t_1=t_2=t_3=s(A)$.  The subgroup $D$ of these diagonal affine maps projects isomorphically to $\operatorname{GL}_3(2)$: it maps onto that group, and a common translation in $U$ must be zero.  Thus $T_0=U:D$.  Pure block permutations preserve \eqref{eq:golay-polynomial-form} and commute with $D$.  Every element of $T$ is a block permutation followed by an element of $T_0$, so
\[
 T=U:(D\times S_3).
\]
This derives the standard trio-group structure directly in the chosen Golay coordinates; compare \cite[Ch.~11, Sec.~12]{ConwaySloane}.

By \eqref{eq:full-affine-eight-action} and Proposition~\ref{prop:deletion-data-transitive}, $D\cong L_3(2)$ acts faithfully and transitively on $\Omega$, with point stabilizer of order $168/8=21$.  Under the exceptional isomorphism $L_3(2)\cong L_2(7)$, such a subgroup is a projective-point stabilizer: its Sylow $7$-subgroup is normal, and the normalizer of a Sylow $7$-subgroup in $L_2(7)$ is the group $7:3$ fixing its unique point on $\mathbf P^1(\F_7)$.  This proves the asserted equivalence of the two degree-eight actions.  The action is doubly transitive, and its centralizer in $\operatorname{Sym}(\Omega)$ is trivial: a commuting permutation must carry a point to another point fixed by its stabilizer, which fixes only that point.  Since the block-permuting $S_3$ commutes with $D$, it acts trivially on $\Omega$.  Formula~\eqref{eq:trio-orbit-stabilizer} follows.

Fix a datum representing $\omega$ and use Corollary~\ref{cor:line-marked-correspondence} to identify its stabilizer with $G_\ell$.  Projection to coordinate permutations has kernel the stabilizer of that datum in the root sign group.  A sign orbit has size $2^{19}$, so this kernel is elementary abelian of order $2^{24-19}=2^5$.  Its image is exactly $T_\omega$: a coordinate permutation preserving the sign orbit can be corrected by a root sign change to fix the datum.  This proves \eqref{eq:line-stabilizer-exact}.  Taking the inverse image of $U$ gives \eqref{eq:line-parabolic-coarse} and
\[
 |G_\ell|=2^5\cdot8064=258048.
\]
The $S_3$ factor permutes the three octads and hence realizes the local action already obtained in Theorem~\ref{thm:positive-definite-uniqueness}.  Solvability follows from \eqref{eq:line-parabolic-coarse}.
\end{proof}

Thus the normal $2$-subgroup of order $2^{11}$ and the quotient $(7:3)\times S_3$ are already visible in the lattice data.  We have not determined the isomorphism type of $P_\ell$ or asserted that the extension by it splits.  The image of $G_\ell$ in the Mathieu trio stabilizer is precisely the index-eight subgroup $T_\omega$, not the whole stabilizer.  The factor $64512$ in the framed count measures all changes of framing, rather than a quotient of $G_\ell$.

\subsection{Group identification and local structure}
\label{sec:triality-identification}

We now identify the group whose order has just been computed.  Steinberg's construction starts with the split group $\mathbf D_4$ and an order-three graph automorphism $\rho$.  If $F_q$ is Frobenius, the fixed-point group of $\rho F_q$ is ${}^3D_4(q)$ \cite{Steinberg1959,CarterLie}.  Tits' triality construction gives its rank-two building, a generalized hexagon of order $(q,q^3)$ in our point-line convention \cite{Tits1959,TitsBuildings}.  At $q=2$ this supplies a generalized hexagon of order $(2,8)$.

By Corollary~\ref{cor:hexagon-uniqueness}, the recovered hexagon is isomorphic to this standard one.  Its full collineation group is ${}^3D_4(2):3$ \cite[p.~89]{Atlas}.  Thus
\begin{equation}\label{eq:AutG-final-id}
 G\cong\Aut(H)\cong{}^3D_4(2):3,
\end{equation}
and the earlier abstract group correspondences give
\begin{equation}\label{eq:AutL-final-id}
 \Aut(L)\cong C_2\times({}^3D_4(2):3),
 \qquad
 \operatorname{Stab}_{F_4(\mathbb R)}(\mathcal D_0)
          \cong{}^3D_4(2):3.
\end{equation}
The normal subgroup ${}^3D_4(2)$ has index $3$ and order $2^{12}3^4 7^2 13$, in agreement with the lattice count.

\begin{corollary}[automorphisms of the rank-$27$ extension]\label{cor:rank27-automorphisms}
Restriction to $c^\perp$ gives natural isomorphisms
\[
 \Aut(M,c)\cong\Aut^+(L),\qquad
 \Aut(M)\cong\Aut(L)\cong C_2\times({}^3D_4(2):3).
\]
\end{corollary}

\begin{proof}
By Remark~\ref{rem:rank27-comparison}, the vectors $\pm c$ are the only characteristic vectors of norm $3$ in $M$.  Every automorphism therefore preserves $\{\pm c\}$ and restricts to an automorphism of $L=c^\perp\cap M$.  If that restriction is the identity, the only possible non-trivial ambient action would fix $L\otimes\R$ and negate $c$.  It does not preserve the gluing \eqref{eq:rank27-glue}: it would send $x+c/3$ to $x-c/3$ for $x\in C_+$, whereas the latter $c$-coordinate is paired with $C_-$.  Thus restriction is injective.

Conversely, if $g\in\Aut(L)$ acts by $\varepsilon\in\{\pm1\}$ on $L^\#/L$, extend it by $c\mapsto\varepsilon c$.  Formula~\eqref{eq:rank27-glue} shows that this extension preserves $M$.  This proves surjectivity and the assertion for the subgroup fixing $c$.  The final group name is \eqref{eq:AutL-final-id}, agreeing with Borcherds' rank-$27$ description \cite[Sec.~5.7]{BorcherdsThesis}.
\end{proof}

The oriented pair $(M,c)$ also recovers the integral generation in the Albert-algebra description from the introduction.  The sign is fixed by \eqref{eq:rank27-glue}: apply Corollary~\ref{cor:oriented-albert} to the half-shell $-X$ and take $E=c$.  The distinguished primitive idempotents are then
\[
 p_x=\frac c3-\frac x2
     =\frac{c-(x+c/3)}2\qquad(x\in X),
 \qquad M=\langle c,\,2p_x\ (x\in X)\rangle_\Z.
\]
Indeed, $2p_x=c-(x+c/3)\in M$, and $c$ together with the vectors $x+c/3$ generates $M$ because $X$ generates $L^\#$ and its augmentation kernel is $L$ (Propositions~\ref{prop:index-three-lattice} and~\ref{prop:inverse}).  A line again gives a Jordan frame with sum $c$.  This concerns the real Albert product and integral generation, not closure of $M$ under Jordan multiplication.

For the finer local group structure, the two parabolic stabilizers have the standard shapes
\[
 G_p\sim 2^{1+8}_{+}:\bigl(L_2(8):3\bigr),
 \qquad
 G_\ell\sim 2^2.[2^9]:\bigl((7:3)\times S_3\bigr)
\]
\cite[p.~89]{Atlas}.  Here $\sim$ denotes group shape, and $[2^9]$ denotes a $2$-group of that order, not an elementary abelian group.  For the line stabilizer this refines \eqref{eq:line-parabolic-coarse}; the finer $2$-group and extension structure in this display is imported from the standard group description.  The $S_3$ action on the three points of the line was obtained directly from the three octads in Theorem~\ref{thm:positive-definite-uniqueness}.  The flag stabilizer is the corresponding index-three subgroup fixing one of them.

\begin{remark}[the two roles of the group calculations]
Theorems~\ref{thm:positive-definite-uniqueness} and~\ref{thm:rank24-order-uniqueness} derive existence, uniqueness, and the order from positive-definite lattices and finite Golay gluing data.  Flag transitivity is already part of Theorem~\ref{thm:positive-definite-uniqueness}, and Proposition~\ref{prop:eight-point-parabolic} gives the coarse line-stabilizer structure.  The triality comparison identifies the group and its finer local structure afterward.  Neither argument invokes Borcherds' Lorentzian uniqueness theorem or the Cohen--Tits uniqueness theorem for the hexagon.  The order also agrees with Borcherds' independent computation \cite[Sec.~5.7]{BorcherdsThesis} and with the rank-$27$ neighbour and enumeration results discussed in the introduction \cite{BacherVenkov,King,ChenevierHunting}.
\end{remark}

The ancillary script \texttt{hexagon\_exact\_checks.py} verifies the displayed intersection tables, eigenmatrix and Gegenbauer sums, the finite Golay actions, and the theta-character recursion using exact arithmetic.  It is a verification aid; none of the proofs depends on running it.

\section*{Acknowledgements}

The author thanks Richard Borcherds for helpful discussions concerning the rank-$26$, determinant-$3$ lattice considered in this paper, which he isolated in his thesis.

\end{document}